\documentclass{amsart}

\usepackage{thesis} 

\begin{document}

\title{\MakeUppercase{\fundthmtitle}}
\date{\today}
\author{\textsc{Aaron Mazel-Gee}}

\begin{abstract}
We prove that a model structure on a relative $\infty$-category $(\M,\bW)$ gives an efficient and computable way of accessing the hom-spaces $\hom_{\loc{\M}{\bW}}(x,y)$ in the localization.  More precisely, we show that when the source $x \in \M$ is \textit{cofibrant} and the target $y \in \M$ is \textit{fibrant}, then this hom-space is a ``quotient'' of the hom-space $\hom_\M(x,y)$ by either of a \textit{left homotopy relation} or a \textit{right homotopy relation}.
\end{abstract}

\maketitle

\papernum{6}

\setcounter{tocdepth}{1}
\tableofcontents

\setcounter{section}{-1}

\section{Introduction}

\subsection{Model $\infty$-categories}\label{subsection intro to fundthm}

\boilerplateintro{\M} The purpose of this paper is to show that the additional data of a \textit{model structure} on $(\M,\bW)$ makes it far easier: we prove the following \bit{fundamental theorem of model $\infty$-categories}.\footnote{For the precise definition a model $\infty$-category, we refer the reader to \cite[\sec 1]{MIC-sspaces}.  However, for the present discussion, it suffices to observe that it is simply a direct generalization of the standard definition of a model category.}

\begin{thm*}[\ref{fundamental theorem}]
Suppose that $\M$ is a model $\infty$-category.  Then, for any cofibrant object $x \in \M^c$ and any fibrant object $y \in \M^f$, the induced map
\[ \hom_\M(x,y) \ra \hom_{\loc{\M}{\bW}}(x,y) \]
on hom-spaces is a $\pi_0$-surjection.  Moreover, this becomes an equivalence upon imposing either of a ``left homotopy relation'' or a ``right homotopy relation'' on the source (see \cref{define spaces of left and right htpy classes of maps}).
\end{thm*}

We view this result -- and the framework of model $\infty$-categories more generally -- as providing a theory of \bit{resolutions} which is native to the $\infty$-categorical setting.  To explain this perspective, let us recall
Quillen's classical theory of model categories, in which for instance
\begin{itemizesmall}
\item replacing a topological space by a CW complex constitutes a \textit{cofibrant resolution} -- that is, a choice of representative which is ``good for mapping out of'' -- of its underlying object of $\Top[\bW_\whe^{-1}]$ (i.e.\! its underlying weak homotopy type), while
\item replacing an $R$-module by a complex of injectives constitutes a \textit{fibrant resolution} -- that is, a choice of representative which is ``good for mapping into'' -- of its underlying object of $\Ch(R)[\bW_\qi^{-1}]$.
\end{itemizesmall}
Thus, a model structure on a relative (1- or $\infty$-)category $(\M,\bW)$ provides \textit{simultaneously compatible} choices of objects of $\M$ which are ``good for mapping out of'' and ``good for mapping into'' with respect to the corresponding localization $\M \ra \loc{\M}{\bW}$.

A prototypical example of this phenomenon arises from the interplay of \bit{left and right derived functors} (in the classical model-categorical sense), i.e.\! of left and right adjoint functors of $\infty$-categories.  For instance,
\begin{itemizesmall}
\item in a \textit{left localization} adjunction $\C \adjarr \leftloc \C$, we can think of the subcategory $\leftloc \C \subset \C$ as that of the ``fibrant'' objects, while \textit{every} object is ``cofibrant'', while dually
\item in a \textit{right localization} adjunction $\rightloc \C \adjarr \C$, we can think of the subcategory $\rightloc \C \subset \C$ as that of the ``cofibrant'' objects, while \textit{every} object is ``fibrant''.\footnote{See \cite[Examples 2.12 and 2.17]{MIC-sspaces} for more details on such model structures.}  
\end{itemizesmall}
As a model structure generally has neither all its objects cofibrant nor all its objects fibrant, it can therefore be seen as a \textit{simultaneous generalization} of the notions of left localization and right localization.
	
\begin{rem}
Indeed, this observation encompasses one of the most important examples of a model $\infty$-category, which was in fact the original motivation for their theory.

Suppose we are given a presentable $\infty$-category $\C$ along with a set $\G$ of generators which we assume (without real loss of generality) to be closed under finite coproducts.  Then, the corresponding \bit{nonabelian derived $\infty$-category} is the $\infty$-category $\PS(\G) = \Fun_\Sigma(\G^{op},\S)$ of those presheaves on $\G$ that take finite coproducts in $\G$ to finite products in $\S$.  This admits a canonical projection
\[ \begin{tikzcd}[row sep=0cm, column sep=2.5cm]
& s\C \arrow{ld}[swap, sloped, pos=0.18]{\hom_\C^\lw(=,-)} \\
s ( \PS(\G) ) \arrow{rd}[swap, sloped, pos=0.62]{|{-}|} \\
& \PS(\G) ,
\end{tikzcd} \]
the composition of the (restricted) levelwise Yoneda embedding (a \textit{right} adjoint) followed by (pointwise) geometric realization (a \textit{left} adjoint): given a simplicial object $Y_\bullet \in s\C$ and a generator $S^\beta \in \G$, this composite is given by
\[ \begin{tikzcd}[row sep=-0.05cm, column sep=1.75cm]
& Y_\bullet \arrow[maps to]{ld} \\
\hom^\lw_\C(S^\beta , Y_\bullet) \arrow[maps to]{rd} \\
& \left| \hom^\lw_\C(S^\beta , Y_\bullet) \right| ,
\end{tikzcd} \]
where we use the abbreviation ``lw'' to denote ``levelwise''.  In fact, this composite is a \textit{free} localization (but neither a left nor a right localization): denoting by $\bW_\res \subset s\C$ the subcategory spanned by those maps which it inverts, it induces an equivalence
\[ \loc{s\C}{\bW_\res} \xra{\sim} \PS(\G) . \]
In future work, we will provide a \bit{resolution model structure} on the $\infty$-category $s\C$ in order to organize computations in the nonabelian derived $\infty$-category $\PS(\G)$.  (The resolution model structure on the $\infty$-category $s\C$, which might also be called an ``$\Etwo$ model structure'', is based on work of Dwyer--Kan--Stover and Bousfield (see \cite{DKS-E2} and \cite{BousCosimp}, resp.).)
\end{rem}

\begin{rem}
In turn, the original motivation for the resolution model structure was provided by \textit{Goerss--Hopkins obstruction theory} (see \cite[\sec 0.3]{MIC-sspaces}).  However, the nonabelian derived $\infty$-category also features prominently for instance in Barwick's universal characterization of \textit{algebraic K-theory} (see \cite{BarwickAKT}), as well as in his theory of \textit{spectral Mackey functors} (which provide an $\infty$-categorical model for genuine equivariant spectra) (see \cite{BarwickMackeyI}).
\end{rem}

\subsection{Conventions}

\refMIC

\tableofcodenames

\examplecodename

\citelurie \ \luriecodenames

\butinvt \ \seeappendix

\subsection{Outline}

We now provide a more detailed outline of the contents of this paper.

\begin{itemize}

\item In \cref{section fundamental theorem}, we give a precise statement of the \bit{fundamental theorem of model $\infty$-categories} (\ref{fundamental theorem}).  This involves the notions of a \textit{cylinder object} $\cyl^\bullet(x) \in c\M$ and a \textit{path object} $\pth_\bullet(y) \in s\M$ for our chosen source and target objects $x,y \in \M$, which generalize their corresponding model 1-categorical namesakes and play analogous roles thereto.

\item In \cref{section left htpy equivt to bisimp colimit}, we prove that the spaces of \textit{left homotopy classes of maps} (defined in terms of a cylinder object $\cyl^\bullet(x)$) and of \textit{right homotopy classes of maps} (defined in terms of a path object $\pth_\bullet(y)$) are both equivalent to a more symmetric bisimplicial colimit (defined in terms of both $\cyl^\bullet(x)$ and $\pth_\bullet(y)$).

\item In \cref{section reduction to special}, we prove that it suffices to consider the case that our cylinder and path objects are \textit{special}.

\item In \cref{section model diagrams and left homotopies}, we digress to introduce \textit{model diagrams}, which corepresent diagrams in a model $\infty$-category $\M$ of a specified type (i.e.\! whose constituent morphisms can be required to be contained in (one or more of) the various defining subcategories $\bW,\bC,\bF \subset \M$).

\item In \cref{section bisimp colimit and special-3}, we prove that when our cylinder and path objects are both special, the bisimplicial colimit of \cref{section left htpy equivt to bisimp colimit} is equivalent to the groupoid completion of a certain $\infty$-category $\tilde{\word{3}}(x,y)$ of \textit{special three-arrow zigzags} from $x$ to $y$.

\item In \cref{section special-3 to 3}, we prove that the inclusion $\tilde{\word{3}}(x,y) \hookra \word{3}(x,y)$ into the $\infty$-category of (all) three-arrow zigzags from $x$ to $y$ induces an equivalence on groupoid completions.

\item In \cref{section 3 to 7}, we prove that the inclusion $\word{3}(x,y) \hookra \word{7}(x,y)$ into a certain $\infty$-category of \textit{seven-arrow zigzags} from $x$ to $y$ induces an equivalences on groupoid completions.

\item In \cref{section model infty-cat gives Ss}, in order to access the hom-spaces in the localization $\loc{\M}{\bW}$, we prove that the \textit{Rezk nerve} $\NerveRezki(\M,\bW)$ (see \cref{rnerves:section rezk nerve}) of (the underlying relative $\infty$-category of) a model $\infty$-category is a Segal space.  (By the local universal property of the Rezk nerve (\cref{rnerves:rezk nerve of a relative infty-category is initial}), this Segal space necessarily presents the localization $\loc{\M}{\bW}$.)

\item In \cref{section equivalence of 7 with hom in loc}, we prove that the groupoid completion $\word{7}(x,y)^\gpd$ of the $\infty$-category of seven-arrow zigzags from $x$ to $y$ is equivalent to the hom-space $\hom_{\loc{\M}{\bW}}(x,y)$.

\item In \cref{section model infty-cat gives cSs}, using the fundamental theorem of model $\infty$-categories (\ref{fundamental theorem}), we prove that the Rezk nerve $\NerveRezki(\M,\bW)$ is in fact a \textit{complete} Segal space.

\end{itemize}

\subsection{Acknowledgments}

It is our pleasure to thank Omar Antol\'{i}n-Camarena, Tobi Barthel, Clark Barwick, Rune Haugseng, Gijs Heuts, Zhen Lin Low, Mike Mandell, Justin Noel, and Aaron Royer for many very helpful conversations.  Additionally, we gratefully acknowledge the financial support provided both by the NSF graduate research fellowship program (grant DGE-1106400) and by UC Berkeley's geometry and topology RTG (grant DMS-0838703) during the time that this work was carried out.

As this paper is the culmination of its series, we would also like to take this opportunity to extend our thanks to the people who have most influenced the entire project, without whom it certainly could never have come into existence: Bill Dwyer and Dans Kan and Quillen, for the model-categorical foundations; Andr\'{e} Joyal and Jacob Lurie, for the $\infty$-categorical foundations; Zhen Lin Low, for countless exceedingly helpful conversations; Eric Peterson, for his tireless, dedicated, and impressive TeX support, and for listening patiently to far too many all-too-elaborate ``here's where I'm stuck'' monologues; David Ayala, for somehow making it through every last one of these papers and providing numerous insightful comments and suggestions, and for providing much-needed encouragement through the tail end of the writing process; Clark Barwick, for many fruitful conversations (including the one that led to the realization that there should exist a notion of ``model $\infty$-categories'' in the first place!), and for his consistent enthusiasm for this project; Peter Teichner, for his trust in allowing free rein to explore, for his generous support and provision for so many extended visits to so many institutions around the world, and for his constant advocacy on our behalf; and, lastly and absolutely essentially, Katherine de Kleer, for her boundless love, for her ample patience, and for bringing beauty and excitement into each and every day.

\section{The fundamental theorem of model $\infty$-categories}\label{section fundamental theorem}

Given an $\infty$-category $\M$ equipped with a subcategory $\bW \subset \M$, the primary purpose of extending these data to a model structure is to obtain an efficient and computable presentation of the hom-spaces in the localization $\loc{\M}{\bW}$.  In this section, we work towards a precise statement of this presentation, which comprises the \textit{fundamental theorem of model $\infty$-categories} (\ref{fundamental theorem}).

A key feature of a model structure is that it allows one to say what it means for two maps in $\M$ to be ``homotopic'', that is, to become equivalent (in the $\infty$-categorical sense) upon application of the localization functor $\M \ra \loc{\M}{\bW}$.  Classically, to pass to the homotopy category of a relative 1-category (i.e.\! to its 1-categorical localization), one simply \textit{identifies} maps that are homotopic.  In keeping with the core philosophy of higher category theory, we will instead want to \textit{remember} these homotopies, and then of course we'll also want to keep track of the higher homotopies between them.

In the theory of model 1-categories, to abstractify the notion of a ``homotopy'' between maps from an object $x$ to an object $y$, one introduces the dual notions of \textit{cylinder objects} and \textit{path objects}.  In the $\infty$-categorical setting, at first glance it might seem that it will suffice to take cylinder and path objects to be as they were before (namely, as certain factorizations of the fold and diagonal maps, respectively): we'll recover a space of maps from a cylinder object for $x$ to $y$, and we might hope that these spaces will keep track of higher homotopies for us.  However, this is not necessarily the case: it might be that a particular homotopy \textit{between} homotopies only exists after passing to a cylinder object on the cylinders themselves.  Of course, it is not possible to guarantee that this process will terminate at some finite stage, and so we must allow for an infinite sequence of such maneuvers.

Although the geometric intuition here no longer corresponds to mere cylinders and paths, we nevertheless recycle the terminology.

\begin{defn}\label{define cyl and path}
Let $\M$ be a model $\infty$-category.  A \bit{cylinder object} for an object $x \in \M$ is a cosimplicial object $\cyl^\bullet(x) \in c\M$ equipped with an equivalence $x \simeq \cyl^0(x)$, such that 
\begin{itemize}
\item the codegeneracy maps $\cyl^n(x) \xra{\sigma^i} \cyl^{n-1}(x)$ are all in $\bW$, and
\item the latching maps $\Latch_n \cyl^\bullet(x) \ra \cyl^n(x)$ are in $\bC$ for all $n \geq 1$.
\end{itemize}
The cylinder object is called \bit{special} if the codegeneracy maps are all also in $\bF$ and the matching maps $\cyl^n(x) \ra \Match_n \cyl^\bullet(x)$ are in $\bW \cap \bF$ for all $n \geq 1$.  We will use the notation $\scyl^\bullet(x) \in c\M$ to denote a special cylinder object for $x \in \M$.

Dually, a \bit{path object} for an object $y \in \M$ is a simplicial object $\pth_\bullet(y) \in s\M$ equipped with an equivalence $y \simeq \pth_0(y)$, such that
\begin{itemize}
\item the degeneracy maps $\pth_n(y) \xra{\sigma_i} \pth_{n+1}(y)$ are all in $\bW$, and
\item the matching maps $\pth_n(y) \ra \Match_n \pth_\bullet(y)$ are in $\bF$ for all $n \geq 1$.
\end{itemize}
The path object is called \bit{special} if the degeneracy maps are all also in $\bC$ and the latching maps $\Latch_n\pth_\bullet(y) \ra \pth_n(y)$ are in $\bW \cap \bC$ for all $n \geq 1$.  We will use the notation $\spth_\bullet(y) \in s\M$ to denote a special path object for $y \in \M$.
\end{defn}

\begin{rem}\label{rem cyl truncates to classical cyl}
Restricting a cylinder object $\cyl^\bullet(x) \in c\M$ to the subcategory $\bD_{\leq 1} \subset \bD$ and employing the identification $x \simeq \cyl^0(x)$, we recover the classical notion of a cylinder object, i.e.\! a factorization
\[ x \sqcup x \cofibn \cyl^1(x) \we x \]
of the fold map; the specialness condition then restricts to the single requirement that the weak equivalence $\cyl^1(x) \we x$ also be a fibration.  In particular, if $\ho(\M)$ is a model category -- recall from \cref{sspaces:model structure on homotopy category} that this will be the case as long as $\ho(\M)$ satisfies limit axiom {\limitaxiom} (i.e.\! is finitely bicomplete), e.g.\! if $\M$ is itself a 1-category --, then a cylinder object $\cyl^\bullet(x) \in c\M$ for $x \in \M$ gives rise to a cylinder object for $x \in \ho(\M)$ in the classical sense.  Of course, dual observations apply to path objects.
\end{rem}

\begin{rem}\label{rem cohypercover}
One might think of a cylinder object as a ``cofibrant $\bW$-cohypercover'', and dually of a path object as a ``fibrant $\bW$-hypercover''.  Indeed, if $x \in \M^c$ then a cylinder object $\cyl^\bullet(x) \in c\M$ defines a cofibrant replacement
\[ \es_{c\M} \cofibn \cyl^\bullet(x) \we \const(x) \]
in $c\M_\Reedy$, and dually if $y \in \M^f$ then a path object $\pth_\bullet(y) \in s\M$ defines a fibrant replacement
\[ \const(y) \we \pth_\bullet(y) \fibn \pt_{s\M} \]
in $s\M_\Reedy$.\footnote{Since the object $[0] \in \bD$ is terminal we obtain an adjunction $(-)^0 : c\M \adjarr \M : \const$, via which the equivalence $\cyl^0(x) \xra{\sim} x$ in $\M$ determines a map $\cyl^\bullet(x) \ra \const(x)$ in $c\M$; the map $\const(y) \ra \pth_\bullet(y)$ arises dually.}  Note, however, that under \cref{define cyl and path}, not every such co/fibrant replacement defines a cylinder/path object, simply because of our requirements that the $0\th$ objects remain unchanged.  In turn, we have made this requirement so that \cref{rem cyl truncates to classical cyl} is true, i.e.\! so that our definition recovers the classical one.

By contrast, in \cite[4.3]{DKFunc}, Dwyer--Kan introduce the notions of ``co/simplicial resolutions'' of objects in a model category (with the ``special'' condition appearing in \cite[Remark 6.8]{DKFunc}).  These are functionally equivalent to our cylinder and path objects; the biggest difference is just that the $0\th$ object of one of their resolutions is required to be a co/fibrant replacement of the original object. 
Of course, we'll ultimately only care about cylinder objects for cofibrant objects and path objects for fibrant objects, and on the other hand they eventually reduce their proofs to the case of co/simplicial resolutions in which this replacement map is the identity (so that in particular the original object is co/fibrant).  Thus, in the end the difference is almost entirely aesthetic. 
\end{rem}

\begin{rem}
Since \cref{define cyl and path} is somewhat involved, here we collect the intuition and/or justification behind each of the pieces of the definition, focusing on (special) path objects.
\begin{itemize}
\item A path object is supposed to be a sort of simplicial resolution.  Thus, the first demand we should place on this simplicial object is that it be ``homotopically constant'', i.e.\! its structure maps should be weak equivalences.  This is accomplished by the requirement that the degeneracy maps lie in $\bW \subset \M$.
\item On the other hand, a path object should also be ``good for mapping into'' (as discussed in \cref{rem cohypercover}).  This fibrancy-like property is encoded by the requirement that the matching maps lie in $\bF \subset \M$.  (By the dual of \cref{cyl for cofibt consists of cofibts} (whose proof uses (the dual of) this condition), when $y \in \M$ is fibrant then so are all the objects $\pth_n(y) \in \M$, for any path object $\pth_\bullet(y) \in s\M$.)
\item The first condition for the specialness of $\pth_\bullet(y)$ -- that the degeneracy maps are (acyclic) cofibrations -- guarantees that for each $n \geq 0$, the unique structure map $y \simeq \pth_0(y) \ra \pth_n(y)$ is also a cofibration.  This is necessary for \cref{path objects are final} to even make sense, and also appears in the proof of the \textit{factorization lemma} (\ref{factorization lemma}).
\item The second condition for the specialness of $\pth_\bullet(y)$ -- that the latching maps be acyclic cofibrations -- guarantees that special path objects are ``weakly initial'' among all path objects (in a sense made precise in \cref{special resns weakly universal}\ref{spth weakly initial}).
\end{itemize}
\end{rem}

Of course, these notions are only useful because of the following existence result.

\begin{prop}\label{special resns}
Let $\M$ be a model $\infty$-category.
\begin{enumerate}
\item\label{special cyl}
Every object of $\M$ admits a special cylinder object.
\item\label{special path}
Every object of $\M$ admits a special path object.
\end{enumerate}
\end{prop}

\begin{proof}
We only prove part \ref{special path}; part \ref{special cyl} will then follow by duality.  So, suppose we are given any object $y \in \M$.  First, set $\pth_0(y) = y$.  Then, we inductively define $\pth_n(y)$ by taking a factorization
\[ \begin{tikzcd}[column sep=0cm]
\Latch_n \pth_\bullet(y) \arrow{rr} \arrow[dashed, tail]{rd}[sloped, swap, pos=0.6]{\approx} & & \Match_n \pth_\bullet(y) \\
& \pth_n(y) \arrow[two heads, dashed]{ru}
\end{tikzcd} \]
of the canonical map using factorization axiom {\factorizationaxiom}.\footnote{At $n=1$, the map $\Latch_1 \pth_\bullet(y) \ra \Match_1 \pth_\bullet(y)$ is just the diagonal map $y \ra y \times y$.}  As observed in \cref{qadjns:reedy constrns work in infty-cats}, this procedure suffices to define a simplicial object $\pth_\bullet(y) \in s\M$.

Now, by construction, above degree 0 the latching maps are all in $\bW \cap \bC$ while the matching maps are all in $\bF$.  Thus, it only remains to check that the degeneracy maps are all in $\bW \cap \bC$.  For this, note that for any $n \geq 0$, every degeneracy map $\pth_n(y) \xra{\sigma_i} \pth_{n+1}(y)$ factors canonically as a composite
\[ \pth_n(y) \ra \Latch_{n+1} \pth_\bullet(y) \wcofibn \pth_{n+1}(y) \]
in $\M$, where the first map is the inclusion into the colimit at the object
\[ ([n]^\opobj \xra{\sigma_i} [n+1]^\opobj) \in \partial \left( \rvec{\bD^{op}}_{/[n+1]^\opobj} \right) . \]
So, it suffices to show that this first map is also in $\bW \cap \bC$.  This follows from applying \cref{map into colim over reedy poset} to the data of
\begin{itemize}
\item the model $\infty$-category $\M$,
\item the Reedy category $\partial \left( \rvec{\bD^{op}}_{/[n+1]^\opobj} \right)$,
\item the maximal object $([n]^\opobj \xra{\sigma_i} [n+1]^\opobj) \in \partial \left( \rvec{\bD^{op}}_{/[n+1]^\opobj} \right)$, and
\item the composite functor
\[ \partial \left( \rvec{\bD^{op}}_{/[n+1]^\opobj} \right) \hookra \rvec{\bD^{op}}_{/[n+1]^\opobj} \ra \rvec{\bD^{op}} \hookra \bD^{op} \xra{\pth_\bullet(y)} \M . \]
\end{itemize}
Indeed, $\partial \left( \rvec{\bD^{op}}_{/[n+1]^\opobj} \right)$ is a Reedy category equal to its own direct subcategory by \cref{qadjns:latching and matching are reedy}\cref{qadjns:latching stuff in latching and matching are reedy}\cref{qadjns:latching cat is reedy}, and it is clearly a poset.  Moreover, our composite functor satisfies the hypothesis of \cref{map into colim over reedy poset} by \cref{qadjns:latching and matching are reedy}\cref{qadjns:latching stuff in latching and matching are reedy}\cref{qadjns:latching in latching is latching}; in fact, all the latching maps are acyclic cofibrations except for possibly the one at the initial object
\[ ([0]^\opobj \ra [n+1]^\opobj) \in \partial \left( \rvec{\bD^{op}}_{/[n+1]^\opobj} \right) . \]
Therefore, the degeneracy map $\pth_n(y) \xra{\sigma_i} \pth_{n+1}(y)$ is indeed an acyclic cofibration, and hence the object $\pth_n(y) \in s\M$ defines a special path object for an arbitrary object $y \in \M$.
\end{proof}

The proof of \cref{special resns} relies on the following result.

\begin{lem}\label{map into colim over reedy poset}
Let $\M$ be a model $\infty$-category, let $\C$ be a Reedy poset which is equal to its own direct subcategory, and let $m \in \C$ be a maximal element.  Suppose that $\C \xra{F} \M$ is a functor such that for any $c \in \C$ which is incomparable to $m \in \C$ (i.e.\! such that $\hom_\C(c,m) = \es_\Set$), the latching map $\Latch_c F \ra F(c)$ lies in $(\bW \cap \bC) \subset \M$.  Then, the induced map $F(m) \ra \colim_\C(F)$ also lies in $(\bW \cap \bC) \subset \M$.
\end{lem}

\begin{proof}
We begin by observing that for any object $c \in \C$, the forgetful map $\C_{/c} \ra \C$ is actually the inclusion of a full subposet.  Now, writing $\C' = (\C \backslash \{ m \}) \subset \C$, it is easy to see that we have a pushout square
\[ \begin{tikzcd}
\partial ( \C_{/m} ) \arrow{r} \arrow{d} & \C_{/m} \arrow{d} \\
\C' \arrow{r} & \C
\end{tikzcd} \]
in $\Cati$ of inclusions of full subposets.  By Proposition T.4.4.2.2, this induces a pushout square
\[ \begin{tikzcd}
\Latch_m F \arrow{r} \arrow{d} & F(m) \arrow{d} \\
\colim_{\C'} (F) \arrow{r} & \colim_\C (F)
\end{tikzcd} \]
in $\M$ (where the colimits all exist by limit axiom {\limitaxiom}, and where we simply write $F$ again for its restriction to any subposet of $\C$).\footnote{In the statement of Proposition T.4.4.2.2, note that the requirement that one of the maps be a monomorphism (i.e.\! a cofibration in $s\Set_\Joyal$) guarantees that this pushout is indeed a homotopy pushout in $s\Set_\Joyal$ (by the left properness of $s\Set_\Joyal$, or alternatively by the Reedy trick).}  Thus, it suffices to show that the map $\Latch_m F \ra \colim_{\C'}(F)$ lies in $(\bW \cap \bC) \subset \M$, since this subcategory is closed under pushouts.

For this, let us choose an ordering
\[  \C' \backslash \partial(\C_{/m}) = \{ c_1,\ldots, c_k \} \]
such that for every $1 \leq i \leq k$ the object $c_i$ is minimal in the full subposet $\{ c_i,\ldots,c_k \} \subset \C$.\footnote{If the Reedy structure on $\C$ is induced by a degree function $\Nerve(\C)_0 \xra{\deg} \bbN$ (which must be possible by its finiteness), then this can be accomplished simply by requiring that $\deg(c_i) \leq \deg(c_{i+1})$ for all $1 \leq i < k$.}  Let us write
\[ \C_i = ( \partial(\C_{/m}) \cup \{ c_1,\ldots,c_i \}) \subset \C' \]
for the full subposet, setting $\C_0 = \partial(\C_{/m})$ for notational convenience, so that we have the chain of inclusions
\[ \partial(\C_{/m}) = \C_0 \subset \cdots \subset \C_k = \C' . \]
Our requirement on the ordering of the objects $c_i$ guarantees that we have
\[ \partial ( \C_{/c_i}) \subset \C_{i-1} , \]
and from here it is not hard to see that in fact we have a pushout square
\[ \begin{tikzcd}
\partial(\C_{/c_i}) \arrow{r} \arrow{d} & \C_{i-1} \arrow{d} \\
\C_{/c_i} \arrow{r} & \C_i
\end{tikzcd} \]
in $\Cati$ for all $1 \leq i \leq k$, from which by again applying Proposition T.4.4.2.2 we obtain a pushout square
\[ \begin{tikzcd}
\Latch_{c_i}F \arrow{r} \arrow{d} & \colim_{\C_{i-1}}(F) \arrow{d} \\
F(c_i) \arrow{r} & \colim_{\C_i}(F)
\end{tikzcd} \]
in $\M$.  But since $\hom_\C(c_i,m) = \es_\Set$ by assumption, our hypotheses imply that the map $\Latch_{c_i} F \ra F(c_i)$ lies in $(\bW \cap \bC) \subset \M$; since this subcategory is closed under pushouts, it follows that it contains the map $\colim_{\C_{i-1}}(F) \ra \colim_{\C_i}(F)$ as well.  Thus, we have obtained the map $\Latch_mF \ra \colim_{\C'}(F)$ as a composite
\[ \Latch_mF = \colim_{\partial(\C_{/m})}(F) = \colim_{\C_0}(F) \wcofibn \cdots \wcofibn \colim_{\C_k}(F) = \colim_{\C'}(F) \]
of acyclic cofibrations in $\M$, so it is itself an acyclic cofibration.  This proves the claim.
\end{proof}

Now that we have shown that (special) cylinder and path objects always exist, we come to the following key definitions.  These should be expected: taking the \textit{quotient} by a relation in a 1-topos corresponds to taking the \textit{geometric realization} of a simplicial object in an $\infty$-topos.  (Among these, \textit{equivalence} relations then correspond to \textit{$\infty$-groupoid} objects (see Definition T.6.1.2.7).)

\begin{defn}\label{define spaces of left and right htpy classes of maps}
Let $\M$ be a model $\infty$, and let $x,y \in \M$.  We define the space of \bit{left homotopy classes of maps} from $x$ to $y$ with respect to a given cylinder object $\cyl^\bullet(x)$ for $x$ to be
\[ \hom_\M^\lsim(x,y) = \left|\hom^{\lw}_\M( \cyl^\bullet(x),y) \right|. \]
Dually, we define the space of \bit{right homotopy classes of maps} from $x$ to $y$ with respect to a given path object $\pth_\bullet(y)$ for $y$ to be
\[ \hom_\M^\rsim(x,y) = \left| \hom^{\lw}_\M(x,\pth_\bullet(y)) \right| . \]
A priori these spaces depend on the choices of cylinder or path objects, but we nevertheless suppress them from the notation.
\end{defn}

\begin{rem}\label{sspaces giving htpy classes of maps are not infty-gpd objects}
Note that $\hom_\M^\lw(x,\pth_\bullet(y))$ is not itself an $\infty$-groupoid object in $\S$.  To ask for this would be too strict: it would not allow for the ``homotopies between homotopies'' that we sought at the beginning of this section.  (Correspondingly, by Yoneda's lemma this would also imply that $\pth_\bullet(y)$ is itself an $\infty$-groupoid object in $\M$, which is clearly a far stronger condition than the ``fibrant $\bW$-hypercover'' heuristic of \cref{rem cohypercover} would dictate.)
\end{rem}

We can now state the \bit{fundamental theorem of model $\infty$-categories}, which says that under the expected co/fibrancy hypotheses, the spaces of left and right homotopy classes of maps both compute the hom-space in the localization.

\begin{thm}\label{fundamental theorem}
Let $\M$ be a model $\infty$-category, suppose that $x \in \M^c$ is cofibrant and $\cyl^\bullet(x) \in c\M$ is any cylinder object for $x$, and suppose that $y \in \M^f$ is fibrant and $\pth_\bullet(y) \in s\M$ is any path object for $y$.  Then there is a
diagram of equivalences
\[ \begin{tikzcd}
\hom_\M^\lsim(x,y) \arrow{r}{\sim} & \left\| \hom^\lw_\M(\cyl^\bullet(x),\pth_\bullet(y)) \right\| \arrow{d}[sloped, anchor=north]{\sim} & \hom_\M^\rsim(x,y) \arrow{l}[swap]{\sim} \\
& \hom_{\loc{\M}{\bW}}(x,y)
\end{tikzcd} \]
in $\S$.
\end{thm}

\begin{proof}
The horizontal equivalences are proved as \cref{homotopy classes of maps and bisimplicial colimit}\ref{lhpty gives bisimp when target fibt} and its dual.  By \cref{reduce to special case}, it suffices to assume that both $\cyl^\bullet(x)$ and $\pth_\bullet(y)$ are special.  The vertical equivalence is then obtained as the composite of the equivalences
\[ \left\| \hom^\lw_\M ( \scyl^\bullet(x),\spth_\bullet(y)) \right\| \simeq \tilde{\word{3}}(x,y)^\gpd \simeq \word{3}(x,y)^\gpd \simeq \word{7}(x,y)^\gpd \simeq \hom_{\loc{\M}{\bW}}(x,y) \]
(where the as-yet-undefined objects of which will be explained in \cref{notn special 3 and 7} and \cref{define infty-cat of model zigzags}) which are respectively proved as Propositions \ref{bisimplicial colimit and special-3} (\and \ref{reduce to special case}), \ref{connect special-3 to 3}, \ref{connect 3 to 7}, \and \ref{connect hom in loc to 7}.
\end{proof}

\begin{rem}
The proof of the fundamental theorem of model $\infty$-categories (\ref{fundamental theorem}) roughly follows that of \cite[Proposition 4.4]{DKFunc} (and specifically the fix given in \cite[\sec 7]{Mandell} for \cite[7.2(iii)]{DKFunc}).  Speaking ahistorically, the main difference is that we have replaced the ultimate appeal to the hammock localization as providing a model for the hom-space $\hom_{\loc{\M}{\bW}}(x,y)$ with an appeal to the ($\infty$-categorical) Rezk nerve $\NerveRezki(\M,\bW)$, which we will prove (as \cref{rnerve is a SS}) likewise provides a model for this hom-space (by the local universal property of the Rezk nerve (\cref{rnerves:rezk nerve of a relative infty-category is initial})).
\end{rem}

An easy consequence of the fundamental theorem of model $\infty$-categories (\ref{fundamental theorem}) is its ``homotopy'' version.

\begin{cor}\label{htpy version of fundamental theorem}
Let $\M$ be a model $\infty$-category, suppose that $x \in \M^c$ is cofibrant and $\cyl^\bullet(x) \in c\M$ is any cylinder object for $x$, and suppose that $y \in \M^f$ is fibrant and $\pth_\bullet(y) \in s\M$ is any path object for $y$.  Then there is a diagram of isomorphisms
\[ \begin{tikzcd}
\left( \dfrac{[x,y]_\M}{[\cyl^1(x),y]_\M} \right) \arrow{r}{\sim} & {[x,y]_{\loc{\M}{\bW}}} & \left( \dfrac{[x,y]_\M}{[x,\pth_1(y)]_\M} \right) \arrow{l}[swap]{\sim}
\end{tikzcd} \]
in $\Set$.
\end{cor}

\begin{proof}
Observe that we have a commutative square
\[ \begin{tikzcd}
s\S \arrow{r}{\pi_0^\lw} \arrow{d}[swap]{\colim^\S_{\bD^{op}}(-)} & s\Set \arrow{d}{\colim^{\Set}_{\bD^{op}}(-)} \\
\S \arrow{r}[swap]{\pi_0} & \Set
\end{tikzcd} \]
in $\Cati$, since all four functors are left adjoints and the resulting composite right adjoints coincide.  The claim now follows immediately from \cref{fundamental theorem}.
\end{proof}

\begin{rem}
In the particular case that $\M$ is a model 1-category, we obtain equivalences $\ho(\M) \xra{\sim} \M$ and $\ho(\loc{\M}{\bW}) \xra{\sim} \M[\bW^{-1}]$.  Hence, \cref{htpy version of fundamental theorem} specializes to recover the classical fundamental theorem of model categories (see e.g.\! \cite[Theorems 7.4.9 and 8.3.9]{Hirsch}).
\end{rem}

\begin{rem}\label{rem left and right htpy do define equivce relns in homotopy cat}
In contrast with \cref{sspaces giving htpy classes of maps are not infty-gpd objects}, the proof of \cite[Theorem 7.4.9]{Hirsch} carries over without essential change to show that in the situation of \cref{htpy version of fundamental theorem}, the diagram
\[ \begin{tikzcd}[column sep=-0.5cm]
{[\cyl^1(x),y]_\M} \arrow[transform canvas={xshift=0.3em}]{rd} \arrow[transform canvas={xshift=-0.3em}]{rd} & & {[x,\pth_1(y)]_\M} \arrow[transform canvas={xshift=0.3em}]{ld} \arrow[transform canvas={xshift=-0.3em}]{ld} \\
& {[x,y]_\M}
\end{tikzcd} \]
\textit{does} define a pair of equal equivalence relations (in $\Set$).
\end{rem}

\section{The equivalence $\hom^\lsim_\M(x,y) \simeq \left\| \hom^\lw_\M(\cyl^\bullet(x),\pth_\bullet(y)) \right\|$}\label{section left htpy equivt to bisimp colimit}

Without first setting up any additional scaffolding, we can immediately prove the horizontal equivalences of \cref{fundamental theorem}.  The following result is an analog of \cite[Proposition 6.2, Corollary 6.4, and Corollary 6.5]{DKFunc}.

\begin{prop}\label{homotopy classes of maps and bisimplicial colimit}
Let $\M$ be a model $\infty$-category, suppose that $x \in \M^c$ is cofibrant, and let $\cyl^\bullet(x) \in c\M$ be any cylinder object for $x$.
\begin{enumerate}
\item\label{preserve wfibns} The functor
\[ \M \xra{\hom^\lw_\M(\cyl^\bullet(x),-)} s\S \]
sends $(\bW \cap \bF) \subset \M$ into $(\bW \cap \bF)_\KQ \subset s\S$.
\item\label{preserve we betw fibts}  The same functor sends $(\M^f \cap \bW) \subset \M$ into $\bW_\KQ \subset s\S$.
\item\label{lhpty gives bisimp when target fibt} If $y \in \M^f$ is fibrant, then for any path object $\pth_\bullet(y) \in s\M$ for $y$, the canonical map $\const(y) \ra \pth_\bullet(y)$ in $s\M$ induces an equivalence
\[ \left| \hom^\lw_\M(\cyl^\bullet(x),y) \right| \xra{\sim} \left\| \hom^\lw_\M(\cyl^\bullet(x),\pth_\bullet(y)) \right\| . \]
\end{enumerate}
\end{prop}

\begin{proof}
To prove part \ref{preserve wfibns}, we use the criterion of \cref{sspaces:detect acyclic fibrations of sspaces} (that $s\S_\KQ$ has a set of generating cofibrations given by the boundary inclusions $I_\KQ = \{ \partial \Delta^n \ra \Delta^n \}_{n \geq 0}$).  First, note that to say that $x$ is cofibrant is to say that the $0^{\th}$ latching map $\es_\M \simeq \Latch_0 \cyl^\bullet(x) \ra \cyl^0(x) \simeq x$ of $\cyl^\bullet(x) \in c\M$ is also a cofibration.  Then, for any $n \geq 0$, suppose we are given an acyclic fibration $y \wfibn z$ in $\M$ inducing the right map in any commutative square
\[ \begin{tikzcd}
\partial \Delta^n \arrow{r} \arrow{d} & \hom^\lw_\M(\cyl^\bullet(x),y) \arrow{d} \\
\Delta^n \arrow{r} & \hom^\lw_\M(\cyl^\bullet(x),z)
\end{tikzcd} \]
in $s\S$.  This commutative square is equivalent data to that of a commutative square
\[ \begin{tikzcd}
\Latch_n \cyl^\bullet(x) \arrow{r} \arrow[tail]{d} & y \arrow[two heads]{d}[sloped, anchor=south]{\approx} \\
\cyl^n(x) \arrow{r} & z,
\end{tikzcd} \]
in $\M$, and moveover a lift in either one determines a lift in the other.  But the latter admits a lift by lifting axiom {\liftingaxiom}.  Hence, the induced map $\hom^\lw_\M(\cyl^\bullet(x),y) \ra \hom^\lw_\M(\cyl^\bullet(x),z)$ is indeed in $(\bW \cap \bF)_\KQ$.

Next, part \ref{preserve we betw fibts} follows immediately from part \ref{preserve wfibns} and the dual of Kenny Brown's lemma (\Cref{qadjns:kenny brown}).

To prove part \ref{lhpty gives bisimp when target fibt}, note that all structure maps in any path object are weak equivalences, and note also that when $y$ is fibrant, then any path object $\pth_\bullet(y)$ consists of fibrant objects by the dual of \cref{cyl for cofibt consists of cofibts}.  Hence, using
\begin{itemizesmall}
\item Fubini's theorem for colimits,
\item part \ref{preserve we betw fibts}, and
\item the fact that simplicial objects whose structure maps are equivalences must be constant,
\end{itemizesmall}
we obtain the string of equivalences
\begin{align*}
\left\| \hom^\lw_\M(\cyl^\bullet(x),\pth_\bullet(y)) \right\| 
& = \colim_{([m]^\opobj,[n]^\opobj) \in \bD^{op} \times \bD^{op}} \hom_\M(\cyl^m(x),\pth_n(y)) \\
& \simeq \colim_{[n]^\opobj \in \bD^{op}} \left( \colim_{[m]^\opobj \in \bD^{op}} \hom_\M(\cyl^m(x),\pth_n(y)) \right) \\
& = \colim_{[n]^\opobj \in \bD^{op}} \left| \hom^\lw_\M(\cyl^\bullet(x),\pth_n(y)) \right| \\
& \simeq \left| \hom^\lw_\M(\cyl^\bullet(x),\pth_0(y)) \right| \\
& \simeq \left| \hom^\lw_\M(\cyl^\bullet(x),y) \right| ,
\end{align*}
proving the claim.
\end{proof}

We needed the following auxiliary result in the proof of \cref{homotopy classes of maps and bisimplicial colimit}.

\begin{lem}\label{cyl for cofibt consists of cofibts}
If $x \in \M^c$ is cofibrant, then for any cylinder object $\cyl^\bullet(x) \in c\M$ for $x$, for every $n \geq 0$ the object $\cyl^n(x) \in \M$ is cofibrant.
\end{lem}

\begin{proof}
Since $\cyl^0(x) \simeq x$ by definition, the claim holds at $n=0$ by assumption.  For $n \geq 1$, by definition we have a cofibration $\Latch_n \cyl^\bullet(x) \cofibn \cyl^n(x)$, so it suffices to show that the object $\Latch_n \cyl^\bullet(x) \in \M$ is cofibrant.  We prove this by induction: at $n=0$, we have $\Latch_0 \cyl^\bullet(x) = \cyl^0(x) \sqcup \cyl^0(x) \simeq x \sqcup x$, which is cofibrant.

Now, recall that by definition,
\[  \Latch_n \cyl^\bullet(x) = \colim_{\partial \left( \rvec{\bD}_{/[n]}\right)} \cyl^\bullet(x) , \]
i.e.\! the latching object is given by the colimit of the composite
\[ \partial \left( \rvec{\bD}_{/[n]} \right) \hookra \rvec{\bD}_{/[n]} \ra \rvec{\bD} \hookra \bD \xra{\cyl^\bullet(x)} \M . \]
Now, by \cref{qadjns:latching and matching are reedy}\cref{qadjns:latching stuff in latching and matching are reedy}\cref{qadjns:latching cat is reedy}, the latching category $\partial \left( \rvec{\bD}_{/[n]}\right)$ admits a Reedy category structure with fibrant constants, so that we obtain a Quillen adjunction
\[ \colim : \Fun \left( \partial \left( \rvec{\bD}_{/[n]}\right) , \M \right)_\Reedy \adjarr \M : \const \]
(since $\M$ is finitely cocomplete by limit axiom {\limitaxiom}).  Thus, it suffices to check that the above composite defines a cofibrant object of $\Fun \left( \partial \left( \rvec{\bD}_{/[n]} \right) , \M \right)_\Reedy$.  For this, given an object $([m] \hookra [n]) \in \partial\left(\rvec{\bD}_{/[n]}\right)$, by \cref{qadjns:latching and matching are reedy}\cref{qadjns:latching stuff in latching and matching are reedy}\cref{qadjns:latching in latching is latching}, its latching category is given by
\[ \partial \left( \rvec{\partial\left(\rvec{\bD}_{/[n]}\right)}_{/([m] \hookra [n])} \right) \cong \partial \left( \rvec{\bD}_{/[m]}\right) . \]
Hence, the latching map of the above composite at this object simply reduces to the cofibration
\[ \Latch_m \cyl^\bullet(x) \cofibn \cyl^m(x) . \]
Therefore, the above composite does indeed define a cofibrant object of $\Fun\left(\partial\left(\rvec{\bD}_{/[n]}\right),\M\right)_\Reedy$, which proves the claim.
\end{proof}

\section{Reduction to the special case}\label{section reduction to special}

In order to proceed with the string of equivalences in the proof of the fundamental theorem of model $\infty$-categories (\ref{fundamental theorem}), we will need to be able to make the assumption that our cylinder and path objects are \textit{special}.  In this section, we therefore reduce to the special case.

\begin{notn}\label{cats of cyls and paths}
Let $\M$ be a model $\infty$-category.  For any $x \in \M$, we write
\[ \catofcyls{x} \subset \left( c\M \underset{(-)^0,\M,x}{\times} \pt_\Cati \right) \]
for the full subcategory on the cylinder objects for $x$, and we write
\[ \catofpths{x} \subset \left( s\M \underset{(-)_0,\M,x}{\times} \pt_\Cati \right) \]
for the full subcategory on the path objects for $x$.
\end{notn}

We now have the following analog of \cite[Propositions 6.9 and 6.10]{DKFunc}.

\begin{lem}\label{special resns weakly universal}
Suppose that $x \in \M$.
\begin{enumerate}
\item\label{scyl weakly terminal} Every special cylinder object $\scyl^\bullet(x) \in \catofcyls{x}$ is weakly terminal: any $\cyl^\bullet(x) \in \catofcyls{x}$ admits a map
\[ \cyl^\bullet(x) \ra \scyl^\bullet(x) \]
in $\catofcyls{x}$.
\item\label{spth weakly initial} Every special path object $\spth_\bullet(x) \in \catofpths{x}$ is weakly initial: any $\pth_\bullet(x) \in \catofpths{x}$ admits a map
\[ \spth_\bullet(x) \ra \pth_\bullet(x) \]
in $\catofpths{x}$.
\end{enumerate}
\end{lem}

\begin{proof}
We only prove the first of two dual statements.  We will construct the map by induction.  The given equivalences
\[ \cyl^0(x) \simeq x \simeq \scyl^0(x) \]
imply that there is a unique way to begin in degree 0.  Then, assuming the map has been constructed up through degree $(n-1)$, \cref{define cyl and path} and lifting axiom {\liftingaxiom} guarantee the existence of a lift in the commutative rectangle
\[ \begin{tikzcd}
\Latch_n \cyl^\bullet(x) \arrow{r} \arrow[tail]{d} & \Latch_n \scyl^\bullet(x) \arrow{r} & \scyl^n(x) \arrow[two heads]{d}[sloped, anchor=south]{\approx} \\
\cyl^n(x) \arrow{r} \arrow[dashed]{rru} & \scyl^n(x) \arrow{r} & \Match_n \scyl^\bullet(x)
\end{tikzcd} \]
in $\M$, which provides an extension of the map up through degree $n$.
\end{proof}

\begin{lem}\label{maps of resns induce equivs}
Let $\M$ be a model $\infty$-category, let $x \in \M^c$ be cofibrant, let $y \in \M^f$ be fibrant, let $\cyl^\bullet_1(x) \ra \cyl^\bullet_2(x)$ be a map in $\catofcyls{x}$, and suppose that $\pth_\bullet(y) \in \catofpths{y}$.  Then the induced maps
\[ \left| \hom^\lw_\M(\cyl^\bullet_2(x),y) \right| \ra \left| \hom^\lw_\M(\cyl^\bullet_1(x),y) \right| \]
and
\[ \left\| \hom^\lw_\M(\cyl^\bullet_2(x),\pth_\bullet(y)) \right\| \ra \left\| \hom^\lw_\M(\cyl^\bullet_1(x),\pth_\bullet(y)) \right\| \]
are equivalences in $\S$.
\end{lem}

\begin{proof}
By \cref{homotopy classes of maps and bisimplicial colimit}\ref{lhpty gives bisimp when target fibt} and its dual, these data induce a commutative diagram
\[ \begin{tikzcd}[column sep=-1cm]
\left| \hom^\lw_\M(\cyl^\bullet_2(x),y) \right| \arrow{rr} \arrow{d}[sloped, anchor=north]{\sim} & & \left| \hom^\lw_\M(\cyl^\bullet_1(x),y) \right| \arrow{d}[sloped, anchor=south]{\sim} \\
\left\| \hom^\lw_\M(\cyl^\bullet_2(x),\pth_\bullet(y)) \right\| \arrow{rr} & & \left\| \hom^\lw_\M(\cyl^\bullet_1(x),\pth_\bullet(y)) \right\| \\
& \left| \hom^\lw_\M(x,\pth_\bullet(y)) \right| \arrow{lu}[sloped, pos=0.4]{\sim} \arrow{ru}[sloped, swap, pos=0.4]{\sim}
\end{tikzcd} \]
of equivalences in $\S$.
\end{proof}

\begin{prop}\label{reduce to special case}
Let $\M$ be a model $\infty$-category, let $x,y \in \M$, let $\cyl^\bullet(x) \in c\M$ be a cylinder object for $x$, and let $\pth_\bullet(y) \in s\M$ be a path object for $y$.  Then there exist
\begin{itemize}
\item a map $\cyl^\bullet(x) \ra \scyl^\bullet(x)$ to a special cylinder object for $x$, and
\item a map $\pth_\bullet(y) \ra \spth_\bullet(y)$ to a special path object for $y$,
\end{itemize}
such that the induced square
\[ \begin{tikzcd}
\hom^\lw_\M(\scyl^\bullet(x) , \spth_\bullet(y)) \arrow{r} \arrow{d} & \hom^\lw_\M(\scyl^\bullet(x),\pth_\bullet(y)) \arrow{d} \\
\hom^\lw_\M(\cyl^\bullet(x),\spth_\bullet(y)) \arrow{r} & \hom^\lw_\M(\cyl^\bullet(x),\pth_\bullet(y))
\end{tikzcd} \]
in $ss\S$ becomes an equivalence upon applying the colimit functor
\[ ss\S \xra{\|{-}\|} \S . \]
\end{prop}

\begin{proof}
The maps are obtained from \cref{special resns weakly universal}; the claim then follows from \cref{maps of resns induce equivs}.
\end{proof}

\section{Model diagrams and left homotopies}\label{section model diagrams and left homotopies}

In the remainder of the proof of the fundamental theorem of model $\infty$-categories (\ref{fundamental theorem}), it will be convenient to have a framework for corepresenting diagrams of a specified type in our model $\infty$-category $\M$.  This leads to the notion of a \textit{model $\infty$-diagram}, which we introduce and study in \cref{subsection model diagrams}.  Then, in \cref{subsection left homotopies}, we specialize this setup to describe the data that thusly corepresents a ``left homotopy'' in the model $\infty$-category $s\S_\KQ$.  (In fact, in order to be completely concrete and explicit we will further specialize to deal only with \textit{model diagrams} (as opposed to model $\infty$-diagrams), since in the end this is all that we will need.)

\subsection{Model diagrams}\label{subsection model diagrams}

We will be interested in $\infty$-categories of diagrams of a specified shape inside of a model $\infty$-category.  These are corepresented, in the following sense.

\begin{defn}\label{define model infty-diagrams}
A \bit{model $\infty$-diagram} is an $\infty$-category $\D$ equipped with three wide subcategories $\bW,\bC,\bF \subset \D$.  These assemble into the evident $\infty$-category, which we denote by $\Modeli$.  Of course, a model $\infty$-category can be considered as a model $\infty$-diagram.  A \bit{model diagram} is a model $\infty$-diagram whose underlying $\infty$-category is a 1-category.  These assemble into a full subcategory $\Model \subset \Modeli$.
\end{defn}

\begin{rem}
We introduced \textit{model diagrams} in \cite[\cref{adjns:define model diagram}]{adjns}, where we required that the subcategory of weak equivalences satisfy the two-out-of-three property.  As this requirement is superfluous for our purposes, we have omitted it from \cref{define model infty-diagrams}.  (However, the wideness requirement is necessary: it guarantees that a map of model diagrams can take any map to an identity map, which in turn jibes with the requirement that the three defining subcategories of a model $\infty$-category be wide.)
\end{rem}

\begin{rem}\label{rel cats can be model diagrams}
A relative $\infty$-category $(\R,\bW)$ can be considered as a model $\infty$-diagram by taking $\bC = \bF = \R^\simeq$.  In this way, we will identify $\RelCati \subset \Modeli$ and $\RelCat \subset \Model$ as full subcategories.\footnote{This inclusion exhibits $\RelCati$ as a right localization of $\Modeli$.  In fact, $\RelCati$ is also a \textit{left} localization of $\Modeli$ via the inclusion which sets both $\bC$ and $\bF$ to be the entire underlying $\infty$-category, but this latter inclusion will not play any role here.}
\end{rem}

\begin{notn}
In order to disambiguate our notation associated with various model $\infty$-diagrams, we will sometimes decorate them for clarity: for instance, we may write $(\D_1,\bW_1,\bC_1,\bF_1)$ and $(\D_2,\bW_2,\bC_2,\bF_2)$ to denote two arbitrary model $\infty$-diagrams.  (This is consistent with both Notations \Cref{sspaces:notn for decorating model structure data} \and \Cref{rnerves:notn for decorating relcat data}.)
\end{notn}

\begin{rem}\label{most mic axioms via model infty-diagrams}
Among the axioms for a model $\infty$-category, all but limit axiom {\limitaxiom} (so two-out-of-three axiom {\twooutofthreeaxiom}, retract axiom {\retractaxiom}, lifting axiom {\liftingaxiom}, and factorization axiom {\factorizationaxiom}) can be encoded by requiring that the underlying model $\infty$-diagram has the extension property with respect to certain maps of model diagrams.
\end{rem}

Since we will be working with a model $\infty$-category with chosen source and target objects of interest, we also introduce the following variant.

\begin{defn}\label{define doubly-pointed model infty-diagram}
A \bit{doubly-pointed model $\infty$-diagram} is a model $\infty$-diagram $\D$ equipped with a map $\pt_\Modeli \sqcup \pt_\Modeli \ra \D$.  The two inclusions $\pt_\Modeli \hookra \pt_\Modeli \sqcup \pt_\Modeli$ select objects $s,t\in \D$, which we call the \bit{source} and \bit{target}; we will sometimes subscript these to remove ambiguity, e.g.\! as $s_\D$ and $t_\D$.  These assemble into the evident $\infty$-category
\[ \Modelip = (\Modeli)_{(\pt_\Modelp \sqcup \pt_\Modelp) / } . \]
Of course, there is a forgetful functor $\Modelip \ra \Modeli$.  We will often implicitly consider a model $\infty$-diagram equipped with two chosen objects as a doubly-pointed model $\infty$-diagram. 
We write $\Modelp \subset \Modelip$ for the full subcategory of \bit{doubly-pointed model diagrams}, i.e.\! of those doubly-pointed model $\infty$-diagrams whose underlying $\infty$-category is a 1-category.
\end{defn}

\begin{rem}\label{doubly-ptd rel cats can be doubly-ptd model diagrams}
Similarly to \cref{rel cats can be model diagrams}, we will consider $\RelCatip \subset \Modelip$ and $\RelCatp \subset \Modelp$ as full subcategories.
\end{rem}

\begin{notn}\label{par ast for maybe}
In order to simultaneously refer to the situations of unpointed and doubly-pointed model $\infty$-diagrams, we will use the notation $\Modelimp$ (and similarly for other related notations).  When we use this notation, we will mean for the entire statement to be interpreted either in the unpointed context or the doubly-pointed context.  (This is consistent with \cref{hammocks:define maybe-pointed relcats}.)
\end{notn}

It will be useful to expand on \cref{hammocks:define relative word} (in view of \cref{doubly-ptd rel cats can be doubly-ptd model diagrams}) in the following way.

\begin{defn}\label{define model word}
We define a \bit{model word} to be a (possibly empty) word $\word{m}$ in any of the symbols $\any$, $\bW$, $\bC$, $\bF$, $(\bW \cap \bC)$, $(\bW \cap \bF)$ or any of their inverses.  Of course, these naturally define doubly-pointed model diagrams; we continue to employ the convention set in \cref{hammocks:define relative word} that we read our model words \textit{forwards}, so that for instance the model word $\word{m} = [\bC;(\bW \cap \bF)^{-1} ; \any]$ defines the doubly-pointed model diagram
\[ \begin{tikzcd}
s \arrow[tail]{r} & \bullet & \bullet \arrow[two heads]{l}[swap]{\approx} \arrow{r} & t .
\end{tikzcd} \]
We denote this object by $\word{m} \in \Modelp$.  Of course, via \cref{doubly-ptd rel cats can be doubly-ptd model diagrams}, we can consider any relative word as a model word.
\end{defn}

\begin{notn}\label{notn special 3 and 7}
Since they will appear repeatedly, we make the abbreviation $\tilde{\word{3}} = [ (\bW \cap \bF)^{-1} ; \any ; (\bW \cap \bC)^{-1}]$ for the model word
\[ \begin{tikzcd}
s & \bullet \arrow{r} \arrow[two heads]{l}[swap]{\approx} & \bullet & t \arrow[tail]{l}[swap]{\approx}
\end{tikzcd} \]
(which is a variant of \cref{hammocks:notn 3}), and we make the abbreviation $\word{7} = [\bW ; \bW^{-1} ; \bW ; \any ; \bW ; \bW^{-1} ; \bW ]$ for the model word (in fact, relative word)
\[ \begin{tikzcd}
s & \bullet \arrow{l}[swap]{\approx} \arrow{r}{\approx} & \bullet & \bullet \arrow{l}[swap]{\approx} \arrow{r} & \bullet & \bullet \arrow{l}[swap]{\approx} \arrow{r}{\approx} & \bullet & t \arrow{l}[swap]{\approx} .
\end{tikzcd} \]
\end{notn}

We now make rigorous ``the $\infty$-category of (either unpointed or doubly-pointed) $\D$-shaped diagrams in $\M$ (and either natural transformations or natural weak equivalences between them)''.

\begin{notn}\label{enr and tensor model over relcat}
Recall from \cref{rnerves:define internal hom in relcats} that $\RelCati$ is a cartesian closed symmetric monoidal $\infty$-category, with internal hom-object given by
\[ \left( \Fun(\R_1,\R_2)^\Rel , \Fun(\R_1,\R_2)^\bW \right) \in \RelCati \]
for $(\R_1,\bW_1) , (\R_2,\bW_2) \in \RelCati$.  It is not hard to see that $\Modeli$ is enriched and tensored over $(\RelCati,\times)$.  Namely, for any
\[ (\D_1,\bW_1,\bC_1,\bF_1),(\D_2,\bW_2,\bC_2,\bF_2) \in \Modeli , \]
we define
\[ \left( \Fun(\D_1,\D_2)^\Model, \Fun(\D_1,\D_2)^\bW \right) \in \RelCati \]
by setting
\[ \Fun(\D_1,\D_2)^\Model \subset \Fun(\D_1,\D_2) \]
to be the full subcategory on those functors which send the subcategories $\bW_1,\bC_1,\bF_1 \subset \D_1$ into $\bW_2,\bC_2,\bF_2 \subset \D_2$ respectively, and setting
\[ \Fun(\D_1,\D_2)^\bW \subset \Fun(\D_1,\D_2)^\Model \]
to be the (generally non-full) subcategory on the natural weak equivalences; moreover, the tensoring is simply the cartesian product in $\Modeli$ (composed with the inclusion $\RelCati \subset \Modeli$ of \cref{rel cats can be model diagrams}).
\end{notn}

\begin{notn}\label{enr and tensor doubly-ptd model over relcat}
Similarly to Notations \ref{enr and tensor model over relcat} \and \Cref{hammocks:define enrichment and tensoring of doubly-pointed relcats over relcats}, $\Modelip$ is enriched and tensored over $(\RelCati,\times)$.  As for the enrichment, for any
\[ (\D_1,\bW_1,\bC_1,\bF_1),(\D_2,\bW_2,\bC_2,\bF_2) \in \Modelip , \]
in analogy with \cref{hammocks:define enrichment and tensoring of doubly-pointed relcats over relcats} we define the object
\[ \left( \Funp(\D_1,\D_2)^\Model, \Funp(\D_1,\D_2)^\bW \right) = \lim \left( \begin{tikzcd}
& \left( \Fun(\D_1,\D_2)^\Model, \Fun(\D_1,\D_2)^\bW \right) \arrow{d}{(\ev_{s_1},\ev_{t_1})} \\
\pt_\RelCati \arrow{r}[swap]{(s_2,t_2)} & (\D_2 ,\bW_2) \times (\D_2,\bW_2)
\end{tikzcd} \right) \]
of $\RelCati$ (where we write $s_1,t_1 \in \D_1$ and $s_2,t_2 \in \D_2$ to distinguish between the source and target objects).  Then, the tensoring is obtained by taking $(\R,\bW_\R) \in \RelCati$ and $(\D,\bW_\D,\bC_\D,\bF_\D) \in \Modelip$ to the pushout
\[ \colim \left( \begin{tikzcd}
\R \times \{ s , t \} \arrow{r} \arrow{d} & \R \times \D \\
\pt_\Modeli \times \{ s , t\}
\end{tikzcd} \right) \]
in $\Modeli$, with its double-pointing given by the natural map from $\pt_\Modeli \sqcup \pt_\Modeli \simeq \pt_\Modeli \times \{ s , t \}$.
\end{notn}

\begin{rem}\label{FunW not an abuse of notation}
While we are using the notation $\Fun(-,-)^\bW$ both in the context of relative $\infty$-categories and model $\infty$-diagrams, due to the identification $\RelCati \subset \Modeli$ of \cref{rel cats can be model diagrams} this is actually \textit{not} an abuse of notation.  The notation $\Funp(-,-)^\bW$ is similarly unambiguous.
\end{rem}

\begin{notn}\label{tensoring of Model over RelCat}
Similarly to \cref{hammocks:notn for tensoring of either relcat or relcatp over relcat}, we will write
\[ \Modelimp \times \RelCati \xra{- \tensoring - } \Modelimp \]
to denote either tensoring of \cref{enr and tensor model over relcat} or of \cref{enr and tensor doubly-ptd model over relcat} (using the convention of \cref{par ast for maybe}).
\end{notn}

Corresponding to \cref{define model word}, we expand on \cref{hammocks:define zigzags} as follows.

\begin{defn}\label{define infty-cat of model zigzags}
Given a model $\infty$-diagram $\M \in \Modeli$ (e.g.\! a model $\infty$-category) equipped with two chosen objects $x,y \in \M$, and given a model word $\word{m} \in \Modelp$, we define the $\infty$-category of \bit{zigzags} in $\M$ from $x$ to $y$ of type $\word{m}$ to be
\[ \word{m}_\M(x,y) = \Funp(\word{m},\M)^\bW . \]
If the model $\infty$-diagram $\M$ is clear from context, we will simply write $\word{m}(x,y)$.\end{defn}

\begin{defn}
For any model $\infty$-diagram $\M$ and any objects $x,y \in \M$, we will refer to
\[ \tilde{\word{3}}(x,y) = \Funp(\tilde{\word{3}},\M)^\bW \in \Cati \]
as the $\infty$-category of \bit{special three-arrow zigzags} in $\M$ from $x$ to $y$ (which is a variant of \cref{hammocks:defn 3-arrow zigzags}), and we will refer to
\[ \word{7}(x,y) = \Funp(\word{7},\M)^\bW \in \Cati \]
as the $\infty$-category of \bit{seven-arrow zigzags} in $\M$ from $x$ to $y$.
\end{defn}

Now, the reason we are interested in the tensorings of \cref{tensoring of Model over RelCat} is the following construction.

\begin{notn}\label{construct cmp}
We define a functor
\[ \Modelimp \xra{\cmp^\bullet} c \Modelimp \]
by setting
\[ \cmp^\bullet \D = \D \tensoring [\bullet]_\bW \]
for any $\D \in \Modelimp$ (where $[\bullet]_\bW$ denotes the composite $\bD \hookra \Cat \xra{\max} \RelCat \hookra \RelCati$).  Of course, this restricts to a functor
\[ \Modelmp \xra{\cmp^\bullet} c\Modelmp . \]
\end{notn}

\begin{ex}
If we consider $[ \bC ; ( \bW \cap \bF)^{-1} ; \any ] \in \Modelp$, then $[ \bC ; ( \bW \cap \bF)^{-1} ; \any ] \tensoring [2]_\bW \in \Modelp$ is given by
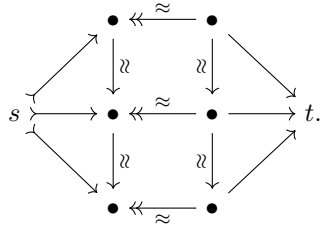
\[ \begin{tikzcd}
& \bullet \arrow{d}[sloped, anchor=south]{\approx} & \bullet \arrow[two heads]{l}[swap]{\approx} \arrow{d}[sloped, anchor=north]{\approx} \arrow{rd} \\
s \arrow[tail]{ru} \arrow[tail]{r} \arrow[tail]{rd} & \bullet \arrow{d}[sloped, anchor=south]{\approx} & \bullet \arrow{r} \arrow[two heads]{l}[swap]{\approx} \arrow{d}[sloped, anchor=north]{\approx} & t . \\
& \bullet & \bullet \arrow[two heads]{l}{\approx} \arrow{ru}
\end{tikzcd} \]
On the other hand, if we consider $[ \bC ; ( \bW \cap \bF)^{-1} ; \any ] \in \Model$, then $[ \bC ; ( \bW \cap \bF)^{-1} ; \any ] \tensoring [2]_\bW \in \Model$ is given by
\[ \begin{tikzcd}
\bullet \arrow[tail]{r} \arrow{d}[sloped, anchor=north]{\approx} & \bullet \arrow{d}[sloped, anchor=north]{\approx} & \bullet \arrow{r} \arrow[two heads]{l}[swap]{\approx} \arrow{d}[sloped, anchor=south]{\approx} & \bullet \arrow{d}[sloped, anchor=south]{\approx} \\
\bullet \arrow[tail]{r} \arrow{d}[sloped, anchor=north]{\approx} & \bullet \arrow{d}[sloped, anchor=north]{\approx} & \bullet \arrow{r} \arrow[two heads]{l}[swap]{\approx} \arrow{d}[sloped, anchor=south]{\approx} & \bullet \arrow{d}[sloped, anchor=south]{\approx} \\
\bullet \arrow[tail]{r} & \bullet & \bullet \arrow{r} \arrow[two heads]{l}{\approx} & \bullet .
\end{tikzcd} \]
\end{ex}

In turn, \cref{construct cmp} is itself useful for the following reason.

\begin{lem}\label{cmp gives nerve of Funmp}
For any $\D , \M \in \Modelimp$, we have an equivalence
\[ \hom^\lw_\Modelimp(\cmp^\bullet \D , \M) \simeq \Nervei \left( \Funmp(\D,\M)^\bW \right) \]
in $s\S$ which is natural in both variables.
\end{lem}

\begin{proof}
For any $n \geq 0$ we have a composite equivalence
\begin{align*}
\Nervei \left( \Funmp(\D,\M)^\bW \right)_n
& = \hom_\Cati \left( [n] , \Funmp(\D,\M)^\bW \right) \\
& \simeq \hom_\RelCati \left( [n]_\bW , \left( \Funmp(\D,\M)^\Model , \Funmp(\D,\M)^\bW \right) \right) \\
& \simeq \hom_\Modelimp ( \D \tensoring [n]_\bW , \M ) \\
& = \hom_\Modelimp ( \cmp^n \D , \M )
\end{align*}
which clearly commutes with the simplicial structure maps on both sides.
\end{proof}

We now introduce slightly more elaborate versions of the concepts we have been exploring -- an $\infty$-categorical version of \cite[\cref{adjns:decorated model diagram}]{adjns} -- which will be used in the proofs of \cref{connect special-3 to 3}, \cref{connect 3 to 7}, and \cref{model infty-cats have calculi}.

\begin{defn}\label{decorated variant}
A \bit{decorated model $\infty$-diagram} is a model $\infty$-diagram with some subdiagrams decorated as colimit or limit diagrams.  For instance, if we define $\D$ to be the ``walking pullback square'', then for any other model $\infty$-diagram $\M$, we let $\hom^\dec_\Modeli(\D,\M) \subset \hom_\Modeli(\D,\M)$, $\Fundec(\D,\M)^\Model \subset \Fun(\D,\M)^\Model$, and $\Fundec(\D,\M)^\bW \subset \Fun(\D,\M)^\bW$ denote the subobjects spanned by those morphisms $\D \ra \M$ of model $\infty$-diagrams which select a pullback square in $\M$.  Of course, we define a \bit{doubly-pointed decorated model $\infty$-diagram} similarly.

In fact, we will only use this variant in the doubly-pointed case, and then only for pushout and pullback squares.  So, in the interest of easing our Ti\textit{k}Zographical burden, we will simply superscript these model diagrams with ``p.o.'' and/or ``p.b.'' as appropriate; the question of which square we are referring to is fully disambiguated by the fact that our pushouts will only be of acyclic cofibrations while our pullbacks will only be of acyclic fibrations.

Note that the constructions $\hom^\dec_\Modelimp(\D,\M) \in \S$ and $\Funmpdec(\D,\M)^\bW \in \Cati$ are not generally functorial in the target $\M$.  On the other hand, they are functorial for \textit{some} maps in the source $\D$.  We will refer to such maps as \bit{decoration-respecting}.  These define an $\infty$-category $\Modelimpdec$.  (Note the distinction between $\hom_\Modelimpdec(-,-)$ and $\hom^\dec_\Modelimp(-,-)$.)  We consider $\Modelimp \subset \Modelimpdec$ simply by considering undecorated model $\infty$-diagrams as being trivially decorated.  We will not need a general theory for understanding which maps of decorated model diagrams are decoration-respecting; rather, it will suffice to observe once and for all that given a square which is decorated as a pushout or pullback square, it is decoration-respecting to either
\begin{itemizesmall}
\item take it to another similarly decorated square, or
\item collapse it onto a single edge (since a commutative square in which two parallel edges are equivalences is both a pushout and a pullback).
\end{itemizesmall}

Note that if the source of a map of decorated model $\infty$-diagrams is actually undecorated, then the map is automatically decoration-respecting; in other words, we must only check that maps in which the \textit{source} is decorated are decoration-respecting.
\end{defn}

\begin{rem}
Of course, adding in \cref{decorated variant} allows us to also demand finite bicompleteness of a model $\infty$-diagram via lifting conditions, and hence all of the axioms for a model $\infty$-diagram to be a model $\infty$-category can now be encoded in this language (recall \cref{most mic axioms via model infty-diagrams}).
\end{rem}

We will need the following analog of \cref{hammocks:nat w.e. induces nat trans} for model $\infty$-diagrams.

\begin{lem}\label{nat w.e. induces nat trans for model diagrams}
Given a pair of maps $\D_1 \rra \D_2$ in $\Modelimpdec$, a morphism between them in $\Funmpdec(\D_1,\D_2)^\bW$ induces, for any $\M \in \Modelimp$, a natural transformation between the two induced functors
\[ \Funmpdec(\D_2,\M)^\bW \rra \Funmpdec(\D_1,\M)^\bW . \]
\end{lem}

\begin{proof}
It is not hard to see that the proof of \cref{hammocks:nat w.e. induces nat trans} carries over without essential change (this time using the enrichment of $\Modelimp$ over $\RelCati$).
\end{proof}

In order to state the final result of this subsection, we need to introduce a bit of notation.

\begin{notn}\label{notn for weak equivces cofibrationing from x and-or fibrationing to y}
For any objects $x,y \in \M$, we denote
\begin{itemizesmall}
\item by
\[ \bW_{x \dcofibn} \subset \bW_{x/} \]
the full subcategory on those objects $(x \wcofibn z) \in \bW_{x/}$ whose structure map is a cofibration,
\item by
\[ \bW_{\dfibn y} \subset \bW_{/y} \]
the full subcategory on those objects $(z \wfibn y) \in \bW_{/y}$ whose structure map is a fibration, and
\item by
\[ \bW_{x \dcofibn \dfibn y} = \bW_{x \dcofibn} \times_\bW \bW_{\dfibn y} \subset \bW_{x/\!/y} \]
the full subcategory on those objects $(x \wcofibn z \wfibn y) \in \bW_{x/\!/y}$ whose structure maps are respectively a cofibration and fibration (as indicated).
\end{itemizesmall}
\end{notn}

We now give an extremely useful result, an analog of \cite[8.1]{DKFunc}, which will appear in the proofs of \cref{connect special-3 to 3}, \cref{connect 3 to 7}, and \cref{model infty-cats have calculi}.  We refer to it as the \bit{factorization lemma}.

\begin{lem}\label{factorization lemma}
Let $\M$ be a model $\infty$-category, and let $x,y \in \M$.  For any model words $\word{m}$ and $\word{n}$, applying $\Funp(-,\M)^\bW$ to the evident inclusion
\[
\left( \begin{tikzcd}[column sep=0.5cm]
s \arrow[-]{rr}{\word{m}} & & \bullet & & \bullet \arrow{ll}[swap]{\approx} \arrow[-]{rr}{\word{n}} & & t
\end{tikzcd} \right)
\ra
\left( \begin{tikzcd}[column sep=0.5cm]
s \arrow[-]{rr}{\word{m}} & & \bullet & & \bullet \arrow{ll}[swap]{\approx} \arrow[-]{rr}{\word{n}} \arrow[tail]{ld}[sloped, pos=0.7]{\approx} & & t \\
& & & \bullet \arrow[two heads]{lu}[sloped, pos=0.3]{\approx}
\end{tikzcd} \right)
\]
in $\Modelp$ induces a map in $\bW_\Thomason \subset \Cati$.
\end{lem}

\begin{proof}
We first observe that the target of this inclusion in $\Modelp$ is isomorphic to the model word
\[ [ \word{m} ; (\bW \cap \bF)^{-1} ; (\bW \cap \bC)^{-1} ; \word{n} ] , \]
it is just drawn so that the ``evident inclusion'' is truly evident.  So, the induced map can be expressed as
\[ [ \word{m} ; (\bW \cap \bF)^{-1} ; (\bW \cap \bC)^{-1} ; \word{n} ](x,y) \ra [\word{m} ; \bW^{-1} ; \word{n}](x,y) . \]
To abbreviate notation, we will write this map in $\Cati$ simply as $\C_1 \ra \C_2$.

Now, showing that the induced map $\C_1^\gpd \ra \C_2^\gpd$ is an equivalence in $\S$ is equivalent to showing that the induced map $(\C_1^{op})^\gpd \ra (\C_2^{op})^\gpd$ is an equivalence in $\S$, and for this by \cref{gr:final functor is an equivalence on groupoid-completions} it suffices to show that the functor $\C_1^{op} \ra \C_2^{op}$ is final.  According to the characterization of Theorem A (\Cref{gr:theorem A}), this is equivalent to showing that for any object
\[ f = \left( \begin{tikzcd}
x \arrow[-]{r}{\word{m}} & x_1 & y_1 \arrow{l}[swap]{\approx} \arrow[-]{r}{\word{n}} & y
\end{tikzcd} \right) \in \C_2 , \]
the groupoid completion of the comma $\infty$-category
\[ (\C_1)^{op} \underset{(\C_2)^{op}}{\times} \left( (\C_2)^{op}\right)_{f^\opobj/} \simeq \left( \C_1 \underset{\C_2}{\times} (\C_2)_{/f} \right)^{op} \]
is contractible, which is in turn equivalent to showing that the groupoid completion of the comma $\infty$-category
\[ \C_3 = \C_1 \underset{\C_2}{\times} (\C_2)_{/f} \]
is contractible.

For this, let us first choose a factorization $y_1 \wcofibn z_1 \wfibn x_1$ in $\M$ using factorization axiom {\factorizationaxiom}; we can consider this as defining an object $Z_1 = (y_1 \wcofibn z_1 \wfibn x_1) \in \M_{y_1 / \! / x_1}$.  Then, working in the model $\infty$-category $\M_{y_1 / \! / x_1}$ (see \cref{sspaces:slice of a model infty-category}), we apply \cref{special resns}\ref{special path} to obtain a special path object $\pth_\bullet(Z_1) \in s(\M_{y_1 / \! / x_1})$.  Note that every constituent object $\pth_n(Z_1) \in \M_{y_1 / \! / x_1}$ is in fact bifibrant: it is cofibrant since specialness implies that the unique structure map $Z_1 \simeq \pth_0(Z_1) \ra \pth_n(Z_1)$ (a composite of degeneracy maps) is an acyclic cofibration and $Z_1$ itself is cofibrant, and it is fibrant by the dual of \cref{cyl for cofibt consists of cofibts} since $Z_1$ itself is fibrant.  Moreover, since $\bW$ has the two-out-of-three property, it follows that in fact $\pth_\bullet(Z_1) \in s(  \bW_{y_1 \dcofibn \dfibn x_1} )$.

Now, observe that there is a natural functor
\[ \bW_{y_1 \dcofibn \dfibn x_1} \ra \C_3 \]
which takes an object $(y_1 \wcofibn w_1 \wfibn x_1) \in \bW_{y_1 \dcofibn \dfibn x_1}$ to the object
\[ \left( \begin{tikzcd}[row sep=0.5cm]
 & x_1 \arrow{dd}[sloped, anchor=south]{\approx} & w_1 \arrow[two heads]{l}[swap]{\approx} & y_1 \arrow[tail]{l}[swap]{\approx} \arrow[-]{rd}[sloped, pos=0.3]{\word{n}} \arrow{dd}[sloped, anchor=north]{\approx} \\
x \arrow[-]{ru}[sloped, pos=0.7]{\word{m}} \arrow[-]{rd}[sloped, swap, pos=0.7]{\word{m}} & & & & y \\
& x_1 & & y_1 \arrow{ll}{\approx} \arrow[-]{ru}[sloped, swap, pos=0.3]{\word{n}}
\end{tikzcd} \right) \in \C_3 \]
(in which diagram the bottom zigzag is the chosen object $f \in \C_2$ and the top zigzag (an object of $\C_1$) is obtained by simply splicing the zigzag $x_1 \lwfibn w_1 \lwcofibn y_1$ into it, and all vertical weak equivalences (including those not pictured) are identity maps).  Thus, we obtain a composite
\[ \bD^{op} \xra{\pth_\bullet(Z_1)} \bW_{y_1 \dcofibn \dfibn x_1} \ra \C_3 , \]
which we will again denote simply by $\pth_\bullet(Z_1) \in s(\C_3)$.  Since $(\bD^{op})^\gpd \simeq \pt_\S$ (as $\bD^{op}$ is sifted), again referring to \cref{gr:final functor is an equivalence on groupoid-completions} we see that it suffices to show that this functor is final.  Then, again referring to Theorem A (\Cref{gr:theorem A}), we see that this is equivalent to showing that for any object
\[ g = \left( \begin{tikzcd}[row sep=0.5cm]
 & x_2 \arrow{dd}[sloped, anchor=south]{\approx} & z_2 \arrow[two heads]{l}[swap]{\approx} & y_2 \arrow[tail]{l}[swap]{\approx} \arrow[-]{rd}[sloped, pos=0.3]{\word{n}} \arrow{dd}[sloped, anchor=north]{\approx} \\
x \arrow[-]{ru}[sloped, pos=0.7]{\word{m}} \arrow[-]{rd}[sloped, swap, pos=0.7]{\word{m}} & & & & y \\
& x_1 & & y_1 \arrow{ll}{\approx} \arrow[-]{ru}[sloped, swap, pos=0.3]{\word{n}}
\end{tikzcd} \right) \in \C_3 \]
(in which diagram the bottom zigzag is again the chosen object $f \in \C_2$ but now the top zigzag is an arbitrary object of $\C_1$), the groupoid completion of the comma $\infty$-category
\[ \C_4 = \bD^{op} \underset{\C_3}{\times} (\C_3)_{g/} \]
is contractible.

For this, let us define a simplicial space $Y \in s\S$ by setting
\[ Y_\bullet = \hom^\lw_{\C_3}(g,\pth_\bullet(Z_1)) . \]
On the one hand, considering $Y \in s\S = \Fun(\bD^{op},\S)$, we have an equivalence
\[ \srep(Y) \simeq \Nervei(\C_4) \]
in $s\S$: for any $n \geq 0$ we have an equivalence
\begin{align*}
\srep(Y)_n & \simeq \coprod_{\alpha \in \Nerve(\bD^{op})_n} Y_{\alpha(0)} \\
& = \coprod_{\alpha \in \Nerve(\bD^{op})_n} \hom_{\C_3}(g,\pth_{\alpha(0)}(Z_1)) \\
& \simeq \Nervei(\C_4)_n ,
\end{align*}
and it is not hard to see that these respect the structure maps of the two simplicial spaces.  But on the other hand, unwinding the definitions we obtain an identification
\[ Y_\bullet \simeq \lim \left( \begin{tikzcd}
& \hom^\lw_{\bW_{/x_1}}(z_2,\pth_\bullet(Z_1)) \arrow{d} \\
\pt_{s\S} \arrow{r} & \hom^\lw_{\bW_{/x_1}}(y_2,\pth_\bullet(Z_1))
\end{tikzcd} \right) , \]
in which pullback
\begin{itemize}
\item we implicitly consider $\pth_\bullet(Z_1) \in s ( \bW_{/x_1})$ via the evident forgetful functor $\bW_{y_1 \dcofibn \dfibn x_1} \ra \bW_{/x_1}$,
\item the vertical map is given by levelwise precomposition with $y_2 \wcofibn z_2$, and
\item the horizontal map is given by the composite
\[ \pt_{s\S} \ra \hom^\lw_{\bW_{/x_1}}(z_1,\pth_\bullet(Z_1)) \ra \hom^\lw_{\bW_{/x_1}}(y_1,\pth_\bullet(Z_1)) \ra \hom^\lw_{\bW_{/x_1}}(y_2,\pth_\bullet(Z_1)) \]
of the canonical point of $\hom^\lw_{\bW_{/x_1}}(z_1,\pth_\bullet(Z_1))$ followed by the maps induced by precomposition with the composite $y_2 \we y_1 \wcofibn z_1$.
\end{itemize}
Considering $\M_{/x_1}$ as a model $\infty$-category (again see \cref{sspaces:slice of a model infty-category}), the simplicial object $\pth_\bullet(Z_1) \in s(\M_{/x_1})$ defines a path object for the fibrant object $z_1 \in (\M_{/x_1})^f$.  Thus, by the dual of \cref{homotopy classes of maps and bisimplicial colimit}\ref{preserve wfibns}, the vertical map in this pullback lies in $(\bW \cap \bF)_\KQ \subset s\S$.  Hence, by \cref{sspaces:the fiber of a fibration of sspaces is homotopically correct} (and \cref{sspaces:detect acyclic fibrations of sspaces}) it follows that $|Y_\bullet| \simeq \pt_\S$.  Finally, combining the two equivalences we have just obtained with the Bousfield--Kan colimit formula (\cref{gr:bousfield--kan}) and \cref{rnerves:groupoid-completion of CSSs}, we obtain the string of equivalences
\[ \pt_\S \simeq |Y_\bullet| \simeq |\srep(Y)_\bullet| \simeq |\Nervei(\C_4)_\bullet| \simeq (\C_4)^\gpd , \]
which completes the proof.
\end{proof}

\subsection{Left homotopies}\label{subsection left homotopies}


Given two parallel maps $\D_1^\bullet \rra \D_2^\bullet$ in $c\Modelmp$, and any $\M \in \Modelmp$, applying the functor
\[ c\Modelmp \xra{ \hom^\lw_\Modelimp(-,\M) } s\S \]
yields two parallel maps
\[ \hom^\lw_\Modelimp(\D_2^\bullet,\M) \rra \hom^\lw_\Modelimp(\D_1^\bullet,\M) \]
in $s\S$.  We will be interested explicitly describing additional data which causes these maps become equivalent upon geometric realization.  This motivates the following definition.

\begin{defn}\label{defn left htpy}
Given two parallel maps $f,g \in \hom_{s\S}(Y,Z)$, a \bit{left homotopy} from $f$ to $g$ (in the model $\infty$-category $s\S_\KQ$) is a map $h \in \hom_{s\S} ( Y \times \Delta^1 , Z)$ fitting into a commutative diagram
\[ \begin{tikzcd}
Y \arrow{r}{\sim} \arrow{rrd}[swap, sloped]{f} & Y \times \Delta^{\{0\}} \arrow{r} & Y \times \Delta^1 \arrow{d}{h} & Y \times \Delta^{\{1\}} \arrow{l} & Y \arrow{l}[swap]{\sim} \arrow{lld}[sloped]{g} \\
& & Z
\end{tikzcd} \]
in $s\S$.
\end{defn}

Of course, this comes with the following expected result.

\begin{lem}\label{left homotopy in sspaces induces equivalence between maps in spaces}
A left homotopy $Y \times \Delta^1 \ra Z$ in $s\S_\KQ$ between two parallel maps $Y \rra Z$ in $s\S$ induces an equivalence between the two induced parallel maps $|Y| \rra |Z|$ in $\S$.
\end{lem}

\begin{proof}
The maps $Y \simeq Y \times \Delta^{\{i\}} \ra Y \times \Delta^1$ are in $\bW_\KQ$ since geometric realization (as a sifted colimit) commutes with finite products.  Hence, the diagram
\[ \begin{tikzcd}
Y \arrow{r}{\sim} \arrow{rrd} & Y \times \Delta^{\{0\}} \arrow{r}{\approx} & Y \times \Delta^1 \arrow{d} & Y \times \Delta^{\{1\}} \arrow{l}[swap]{\approx} & Y \arrow{l}[swap]{\sim} \arrow{lld} \\
& & Z
\end{tikzcd} \]
in $s\S_\KQ$ induces, upon geometric realization, the diagram
\[ \begin{tikzcd}
|Y| \arrow{r}{\sim} \arrow{rrd} & |Y \times \Delta^{\{0\}}| \arrow{r}{\sim} & |Y \times \Delta^1| \arrow{d} & |Y \times \Delta^{\{1\}}| \arrow{l}[swap]{\sim} & |Y| \arrow{l}[swap]{\sim} \arrow{lld} \\
& & |Z|
\end{tikzcd} \]
in $\S$, which selects the desired equivalence between the two induced maps $|Y| \rra |Z|$.
\end{proof}

In our cases of interest, the left homotopy between two parallel maps
\[ \hom^\lw_\Modelimp(\D_2^\bullet,\M) \rra \hom^\lw_\Modelimp(\D_1^\bullet,\M) \]
will be natural in the variable $\M \in \Modelimp$.  By Yoneda's lemma, the data of such a left homotopy itself will be corepresentable by some additional data relating $\D_1^\bullet$ and $\D_2^\bullet$.  This leads us to the following definition.

\begin{defn}
Given $\varphi^\bullet,\psi^\bullet \in \hom_{c\Modelmp}(\D_1^\bullet,\D_2^\bullet)$, a \bit{left homotopy corepresentation} from $\varphi^\bullet$ to $\psi^\bullet$ is a family of maps
\[ \{ h^i_n \in \hom_\Modelmp(\D_1^{n+1},\D_2^n) \}_{0 \leq i \leq n \geq 0} \]
satisfying the identities
\begin{align*}
h^0_n \delta^0 & = \varphi^n \\
h^n_n \delta^{n+1} & = \psi^n \\
h^j_n \delta^i &= \left\{ \begin{array}{ll}
\delta^i h^{j-1}_{n-1}, & i < j \\
h^{j-1}_n \delta^i , & i = j \not= 0 \\
\delta^{i-1} h^j_{n-1} , & i > j+1
\end{array} \right. \\
h^j_n \sigma^i &= \left\{ \begin{array}{ll}
\sigma^j h^{j+1}_{n+1}, & i \leq j \\
\sigma^{i-1} h^j_{n+1} , & i>j .
\end{array} \right.
\end{align*}
\end{defn}

\begin{rem}
These identities are nothing but the duals of those defining a ``simplicial homotopy'' in the classical sense (see e.g.\! \cite[Definitions 5.1]{MaySimp}).
\end{rem}

Then, we have the following expected result.

\begin{lem}\label{left homotopy corepresentation induces left homotopy}
Fix some $\varphi^\bullet ,\psi^\bullet \in \hom_{c\Modelmp}(\D_1^\bullet,\D_2^\bullet)$.  Then, giving a left homotopy corepresentation
\[ \{ h^i_n \in \hom_\Modelmp(\D_1^{n+1},\D_2^n) \}_{0 \leq i \leq n \geq 0} \]
from $\varphi^\bullet$ to $\psi^\bullet$ is equivalent to giving a left homotopy
\[ \hom_\Modelimp^\lw(\D_2^\bullet, \M) \times \Delta^1 \ra \hom_\Modelimp^\lw(\D_1^\bullet, \M) \]
from $\hom_\Modelimp^\lw(\varphi^\bullet , \M)$ to $\hom_\Modelimp^\lw(\psi^\bullet, \M)$ which is natural in the variable $\M \in \Modelimp$.
\end{lem}

\begin{proof}
Suppose we have such a natural left homotopy.  If we apply it to $\D_2^n$, the natural map
\[ \Delta^n \ra \hom^\lw_\Modelmp(\D_2^\bullet,\D_2^n) \]
in $s\S$ corresponding to $\id_{\D_2^n}$ gives rise to the composite map
\[ \Delta^n \times \Delta^1 \ra \hom^\lw_\Modelmp(\D_2^\bullet,\D_2^n) \times \Delta^1 \ra \hom^\lw_\Modelmp(\D_1^\bullet,\D_2^n) . \]
Evaluating this at the $n+1$ nondegenerate $(n+1)$-simplices of $\Delta^n \times \Delta^1$ and ranging over all $n \geq 0$ yields the maps defining the left homotopy corepresentation; that these satisfy the identities follows from applying the natural left homotopy to the cosimplicial structure maps of $\D_2^\bullet \in c\Modelmp$.

Conversely, given a left homotopy representation, we define a natural left homotopy given in level $n$ by the map
\[ \hom_\Modelimp(\D_2^n,\M) \times (\Delta^1)_n \simeq \coprod_{(\Delta^1)_n} \hom_\Modelimp(\D_2^n , \M) \ra \hom_\Modelimp(\D_1^n, \M) \]
which, on the summand corresponding to the element of $(\Delta^1)_n \cong \hom_{\bD}([n],[1])$ associated to the decomposition
\[ [n] = \{ 0 , \ldots, n-i\} \sqcup \{ (n+1)-i,\ldots,n\} \]
(for $i \in \{0,\ldots,n+1\}$), is corepresented by the map
\[ \left\{ \begin{array}{ll}
\varphi^n = h^0_n \delta^0 , & i=0 \\
h^{i-1}_n \delta^i = h^i_n \delta^n , & 0 < i < n+1 \\
\psi^n = h^n_n \delta^{n+1} , & i=n+1
\end{array} \right. \]
in $\hom_\Modelmp(\D_1^n,\D_2^n)$; that these do indeed define a left homotopy follows from the fact that our choices here are induced by the simplicial structure maps of $\Delta^1 \in s\Set \subset s\S$.
\end{proof}

\begin{defn}
In the situation of \cref{left homotopy corepresentation induces left homotopy}, we refer to an induced map
\[ \hom_\Modelimp^\lw(\D_2^\bullet, \M) \times \Delta^1 \ra \hom_\Modelimp^\lw(\D_1^\bullet, \M) \]
as a \bit{corepresented left homotopy} (in the model $\infty$-category $s\S_\KQ$) associated to the left homotopy corepresentation.
\end{defn}

\section{The equivalence $\left\| \hom^\lw_\M(\scyl^\bullet(x),\spth_\bullet(y)) \right\| \simeq \tilde{\word{3}}(x,y)^\gpd$}\label{section bisimp colimit and special-3}

We now proceed with an analog of \cite[Proposition 7.3]{Mandell}.

\begin{prop}\label{bisimplicial colimit and special-3}
Suppose we have $x,y \in \M$ with $x$ cofibrant and $y$ fibrant, and let $\scyl^\bullet(x) \in c\M$ and $\spth_\bullet(y) \in s\M$ be a special cylinder object for $x$ and a special path object for $y$, respectively.  Then
\[ \left\| \hom^\lw_\M(\scyl^\bullet(x),\spth_\bullet(y)) \right\| \simeq \tilde{\word{3}}(x,y)^\gpd. \]
\end{prop}

\begin{proof}
To prove the claim, we construct a commutative diagram
\[ \begin{tikzcd}
M_\bullet \arrow{d} \arrow{r} & Q_\bullet & P_\bullet \arrow{l} \arrow{d} \\
N_\bullet \arrow{ru} & P_\bullet \arrow{l} \arrow{u} \arrow{r} & P_\bullet \times \Delta^1 \arrow{lu}
\end{tikzcd} \]
in $s\S$ whose maps are all in $\bW_\KQ$, such that
\[ |M_\bullet| \simeq \left\| \hom^\lw_\M(\scyl^\bullet(x),\spth_\bullet(y)) \right\| \]
and
\[ |Q_\bullet| \simeq \tilde{\word{3}}(x,y)^\gpd . \]

We first define the simplicial spaces of the diagram.  Certain auxiliary definitions will appear superfluous, but they will be used later in the proof.

\begin{itemize}

\item We begin by defining the object $M_\bullet \in s\S$ by
\[ M_\bullet = \srep \left( \bD^{op} \times \bD^{op} \xra{\hom_\M(\scyl^\bullet(x),\spth_\bullet(y))} \S \right)_\bullet . \]
By the Bousfield--Kan colimit formula (\cref{gr:bousfield--kan}), we have that
\[ |M_\bullet| \simeq \left\| \hom^\lw_\M(\scyl^\bullet(x),\spth_\bullet(y)) \right\| , \]
as desired.  Note that, since $[n] \in \strcat$ and $\bD \times \bD^{op} \in \strcat$ are gaunt, up to making the identification
\[ \hom_\Cati([n] , \bD^{op}) \simeq \hom_\Cati([n]^{op},\bD^{op}) \simeq \hom_\Cati([n],\bD) , \]
we have that
\begin{align*}
M_n & \simeq \colim_{(\alpha,\beta) \in \hom_\Cati([n],\bD \times \bD^{op})} \hom_\M(\scyl^{\alpha(n)}(x),\spth_{\beta(0)}(y)) \\
& \simeq \coprod_{(\alpha,\beta) \in \Nerve(\bD)_n \times \Nerve(\bD^{op})_n} \hom_\M(\scyl^{\alpha(n)}(x),\spth_{\beta(0)}(y)) .
\end{align*}

\item We define the objects $N_\bullet, Q_\bullet , P_\bullet \in s\S$ simultaneously, as follows.  For any $m,n \geq 0$, let $\word{p}^{m,n}$ denote the doubly-pointed model diagram
\[ \begin{tikzcd}[row sep=1.25cm]
& s \\
\alpha(0) \arrow{r}[swap]{\approx} \arrow[two heads]{ru}[sloped, pos=0.6]{\approx} \arrow{d} & \cdots \arrow{r}[swap]{\approx} & \alpha(m) \arrow[two heads]{lu}[sloped, swap, pos=0.6]{\approx} \arrow{d} \arrow{rd} \\
\beta(0) \arrow{r}{\approx} & \cdots \arrow{r}{\approx} & \beta(m) \arrow{r}{\approx} & \gamma(0) \arrow{r}{\approx} & \cdots \arrow{r}{\approx} & \gamma(n) \\
& t . \arrow[tail]{lu}[sloped, pos=0.4]{\approx} \arrow[tail, bend left=0]{ru}[sloped, swap, pos=0.4]{\approx} \arrow[tail, bend left=0]{rru}[sloped, swap]{\approx} \arrow[tail, bend right=20]{rrrru}[sloped,swap]{\approx}
\end{tikzcd} \]

Moreover, let $\word{n}^{m,n} \subset \word{p}^{m,n}$ denote the full subcategory on the objects $\{ s , t, \alpha(i),\gamma(j) \}_{0 \leq i \leq m , 0 \leq j \leq n}$ and let $\word{q}^{m,n} \subset \word{p}^{m,n}$ denote the full subcategory on the objects $\{ s , t, \alpha(i),\beta(j)\}_{0 \leq i,j \leq m}$, both considered as doubly-pointed model diagrams in the evident way.  Let us use the placeholders $Y \in \{ N,Q,P\}$ and $\word{y} \in \{ \word{n},\word{q},\word{p} \}$.  Then, the various objects $\word{y}^{m,n} \in \Modelp$ assemble into the evident bicosimplicial object $\word{y}^{\bullet \bullet} \in c\Modelp$, and we auxiliarily define
\[ Y_{\bullet \bullet} = \hom^\lw_\Modelip(\word{y}^{\bullet\bullet} , \M) \in ss\S . \]
Then, we define $\word{y}^\bullet = \diag^*(\word{y}^{\bullet \bullet}) \in c\Modelp$, and we set
\[ Y_\bullet = \hom^\lw_\Modelip(\word{y}^\bullet,\M) \in s\S , \]
so that $Y_\bullet \simeq \diag^*(Y_{\bullet \bullet})$.

We now provide alternative identifications of the simplicial spaces $N_\bullet$ and $Q_\bullet$.

\begin{itemize}

\item As for $N_\bullet$, we clearly have
\[ N_n \simeq \colim_{(\alpha,\gamma) \in \hom_\Cati \left( [n], \bW_{\dfibn x} \times \bW_{y \dcofibn} \right) } \hom_\M(\alpha(n),\gamma(0)) . \]
Moreover, examining the structure maps of $N_\bullet \in s\S$, we see that up to making the identification
\[ \hom_\Cati \left([n],\left(\bW_{\dfibn x} \right)^{op} \right) \simeq \hom_\Cati \left([n]^{op},\left(\bW_{\dfibn x}\right)^{op}\right) \simeq \hom_\Cati \left([n] , \bW_{\dfibn x} \right) , \]
we have that
\[ N_\bullet \simeq \srep \left( \left( \bW_{\dfibn x} \right)^{op} \times \bW_{y \dcofibn} \xra{\left( (x' \wfibn x)^\opobj , (y \wcofibn y') \right) \mapsto \hom_\M(x',y')} \S \right)_\bullet . \]

\item As for $Q_\bullet$, note first of all that $\word{q}^{m,n} \in \Modelp$ (and hence $Q_{m,n} \in \S$) is independent of $n$.  Moreover, since we have an evident isomorphism $\word{q}^\bullet \cong \cp^\bullet \tilde{\word{3}}$ in $c\Modelp$ -- indeed, the only difference is that we have named the intermediate objects of the constituent model diagrams of $\word{q}^\bullet \in c\Modelp$ -- it follows from \cref{cmp gives nerve of Funmp} that
\[ Q_\bullet \simeq \Nervei(\tilde{\word{3}}(x,y))_\bullet . \]
Hence, \cref{rnerves:groupoid-completion of CSSs} this implies that
\[ |Q_\bullet| \simeq \tilde{\word{3}}(x,y)^\gpd , \]
as desired.

\end{itemize}

Finally, we observe that since $\bD^{op} \xra{\diag} \bD^{op} \times \bD^{op}$ is final (as $\bD^{op}$ is sifted), then by Fubini's theorem for colimits, continuing to use the placeholder $Y \in \{ N,Q,P\}$ we have an identification
\begin{align*}
|Y_\bullet| & \simeq \| Y_{\bullet \bullet} \| \\
& = \colim_{([m]^\opobj,[n]^\opobj) \in \bD^{op} \times \bD^{op}} Y_{m,n} \\
& \simeq \colim_{[n]^\opobj \in \bD^{op}} \left( \colim_{[m]^\opobj \in \bD^{op}} Y_{m,n} \right) \\
& = \colim_{[n]^\opobj \in \bD^{op}} |Y_{\bullet,n}| ,
\end{align*}
and similarly we have an identification
\[ | Y_\bullet | \simeq \colim_{[m]^\opobj \in \bD^{op}} | Y_{m,\bullet} | . \]

\end{itemize}

We now define the maps in the diagram, and along the way we show that the subdiagram
\[ \begin{tikzcd}
M_\bullet \arrow{d} & Q_\bullet & P_\bullet \arrow{l} \arrow{d} \\
N_\bullet & P_\bullet \arrow{l} \arrow{r} & P_\bullet \times \Delta^1
\end{tikzcd} \]
lies in $\bW_\KQ$, which suffices to prove that the entire diagram is in $\bW_\KQ$ by the two-out-of-three property.\footnote{Of course, really it would already have sufficed to obtain the zigzag $M_\bullet \ra N_\bullet \la P_\bullet \ra Q_\bullet$ of maps in $\bW_\KQ$, but this proof is almost no more work and has the added benefit of showing that the map inducing the equivalence is the expected one.}

\begin{itemize}

\item We have a commutative diagram
\[ \begin{tikzcd}[column sep=4cm, row sep=2cm]
\bD^{op} \times \bD^{op} \arrow{rr}{([m]^\opobj,[n]^\opobj) \mapsto \left( (\scyl^m(x) \wfibn x)^\opobj , (y \wcofibn \spth_n(y)) \right)} \arrow{rd}[swap]{\hom^\lw_\M(\scyl^\bullet(x),\spth_\bullet(y))} & & \left( \bW_{\dfibn x} \right)^{op} \times \bW_{y \dcofibn} \arrow{ld}{\left( (x' \wfibn x)^\opobj , (y \wcofibn y') \right) \mapsto \hom_\M(x',y')} \\
& \S
\end{tikzcd} \]
in $\Cati$; considering this as a map in $(\Cati)_{/\S}$, we obtain the map $M_\bullet \ra N_\bullet$ from \cref{gr:bousfield--kan functorial in source for Cati/C}\cref{gr:compatibility of map on sreps}.  The upper map in this diagram is the product of two functors which are each final, the second by \cref{path objects are final} and the first by the opposite of its dual.  Hence, this functor is itself final by \cref{gr:product of final functors is final}.  Thus, the map $M_\bullet \ra N_\bullet$ is in $\bW_\KQ$ by the Bousfield--Kan colimit formula (\cref{gr:bousfield--kan}).

\item The map $N_\bullet \ra Q_\bullet$ is corepresented by the morphism in $\hom_{c\Modelp}(\word{q}^\bullet,\word{n}^\bullet)$ given in level $n$ by the unique functor satisfying $\alpha(i) \mapsto \alpha(i)$ and $\beta(i) \mapsto \gamma(i)$.  (Note that there are composite morphisms $\alpha(i) \ra \beta(i)$ implicit in the diagram defining $\word{n}^n$.)

\item The map $M_\bullet \ra Q_\bullet$ is the composition $M_\bullet \ra N_\bullet \ra Q_\bullet$.

\item The map $P_\bullet \ra N_\bullet$ is corepresented by the morphism in $\hom_{c\Modelp}(\word{n}^\bullet,\word{p}^\bullet)$ which is simply the defining inclusion in each level.  Note that this is obtained by applying $cc\Modelp \xra{\diag^*} c\Modelp$ to the morphism in $\hom_{cc\Modelp}(\word{n}^{\bullet \bullet} , \word{p}^{\bullet \bullet})$ which is again simply the defining inclusion in each bidegree.  This latter map corepresents a map $P_{\bullet \bullet} \ra N_{\bullet \bullet}$ in $ss\S$, from which the map $P_\bullet \ra N_\bullet$ in $s\S$ is therefore obtained by applying $ss\S \xra{\diag^*} s\S$.

Now, since $| P_\bullet | \simeq \colim_{[n]^\opobj \in \bD^{op}} |P_{\bullet,n}|$ and $|N_\bullet| \simeq \colim_{[n]^\opobj \in \bD^{op}} |N_{\bullet,n}|$, to prove that the map $P_\bullet \ra N_\bullet$ is in $\bW_\KQ$, it suffices to prove that for each $[n]^\opobj \in \bD^{op}$, the map $|P_{\bullet,n}| \ra |N_{\bullet,n}|$ is an equivalence in $\S$, i.e.\! that the map $P_{\bullet,n} \ra N_{\bullet,n}$ is in $\bW_\KQ$.

To see this, we construct an inverse up to left homotopy in $s\S_\KQ$ for this map.  This is corepresented by the map in $\hom_{c\Modelp}(\word{p}^{\bullet,n},\word{n}^{\bullet,n})$ given in level $m$ by the unique functor satisfying $\alpha(i) \mapsto \alpha(i)$, $\beta(i) \mapsto \gamma(0)$, and $\gamma(i) \mapsto \gamma(i)$.  As the resulting composite map $\word{n}^{\bullet,n} \ra \word{p}^{\bullet,n} \ra \word{n}^{\bullet,n}$ in $c\Modelp$ is the identity, it follows that the corepresented composite map $N_{\bullet,n} \ra P_{\bullet,n} \ra N_{\bullet,n}$ is also the identity.

On the other hand, the composite map $\word{p}^{\bullet,n} \ra \word{n}^{\bullet,n} \ra \word{p}^{\bullet,n}$ is not equal to the identity.  However, it suffices to give a left homotopy corepresentation
\[ \{ _\word{p}h^i_m \in \hom_\Modelp(\word{p}^{m+1,n},\word{p}^{m,n}) \}_{0 \leq i \leq m \geq 0} \]
from this composite to $\id_{\word{p}^{\bullet, n}}$, which we define by taking $_\word{p}h^i_m$ to be the unique functor satisfying
\begin{align*}
\alpha(j) & \mapsto \left\{ \begin{array}{ll}
\alpha(j), & j \leq i \\
\alpha(j-1), & j > i
\end{array} \right. \\
\beta(j) & \mapsto \left\{ \begin{array}{ll}
\beta(j), & j \leq i \\
\gamma(0), & j > i \\
\end{array} \right. \\
\gamma(j) & \mapsto \gamma(j).
\end{align*}
(It is tedious but straightforward to verify that these formulas do indeed define such a left homotopy corepresentation.)  By \cref{left homotopy corepresentation induces left homotopy} this gives us a left homotopy in $s\S_\KQ$ from the corepresented composite map $P_{\bullet,n} \ra N_{\bullet,n} \ra P_{\bullet,n}$ to $\id_{P_{\bullet,n}}$, and so by \cref{left homotopy in sspaces induces equivalence between maps in spaces} this corepresented composite map becomes equivalent upon geometric realization to $\id_{|P_{\bullet,n}|}$.  Thus, the map $P_{\bullet,n} \ra N_{\bullet,n}$ does indeed lie in $\bW_\KQ$ for all $[n]^\opobj \in \bD^{op}$, so that the map $P_\bullet \ra N_\bullet$ lies in $\bW_\KQ$ as well.

\item The vertical map $P_\bullet \ra Q_\bullet$ is of course given by the composition $P_\bullet \ra N_\bullet \ra Q_\bullet$.  More explicitly, it is corepresented by the morphism in $\hom_{c\Modelp}(\word{q}^\bullet , \word{p}^\bullet)$ given in level $n$ by the unique functor satisfying $\alpha(i) \mapsto \alpha(i)$ and $\beta(i) \mapsto \gamma(i)$.

\item The horizontal map $P_\bullet \ra Q_\bullet$ is corepresented by the morphism in $\hom_{c\Modelp}(\word{q}^\bullet,\word{p}^\bullet)$ which is simply the the defining inclusion in each level.  Note that this is obtained by applying $cc\Modelp \xra{\diag^*} c\Modelp$ to the morphism in $\hom_{cc\Modelp}(\word{q}^{\bullet \bullet},\word{p}^{\bullet \bullet})$ which is again simply the defining inclusion in each bidegree.  This latter map corepresents a map $P_{\bullet \bullet} \ra Q_{\bullet \bullet}$ in $ss\S$, from which the horizontal map $P_\bullet \ra Q_\bullet$ in $s\S$ is therefore obtained by applying $ss\S \xra{\diag^*} s\S$.

Now, since $|P_\bullet| \simeq \colim_{[m]^\opobj \in \bD^{op}} | P_{m,\bullet}|$ and $|Q_\bullet| \simeq \colim_{[m]^\opobj \in \bD^{op}} |Q_{m,\bullet}|$, to prove that the horizontal map $P_\bullet \ra Q_\bullet$ is in $\bW_\KQ$, it suffices to prove that for each $[m]^\opobj \in \bD^{op}$, the map $|P_{m,\bullet}| \ra |Q_{m,\bullet}| \simeq Q_m$ is an equivalence in $\S$ (where the given equivalence comes from the fact that $Q_{m,\bullet} \simeq \const(Q_m)$).

Via the map $P_{m,\bullet} \ra Q_{m,\bullet} \simeq \const(Q_m)$, we can consider $P_{m,\bullet}$ as a simplicial object
\[ \bD^{op} \xra{P_{m,\bullet}} \S_{/Q_m} ; \]
moreover, $|P_{m,\bullet}|$ is still its colimit in this $\infty$-category since colimits in $\S_{/Q_m}$ are created in $\S$.  Now, we have a composite equivalence
\[ \Fun(Q_m,\S) \xra[\sim]{\Gr} \LFib(Q_m) \simeq \S_{/Q_m} \]
(recall \cref{gr:over a space Gr two adjns become equivces}), under which the above simplicial object corresponds to a simplicial object
\[ \bD^{op} \xra{\Gr^{-1}(P_{m,\bullet})} \Fun(Q_m,\S) . \]
Hence, to show that $|P_{m,\bullet}| \in \S_{/Q_m}$ is a terminal object (i.e.\! to show that $|P_{m,\bullet}| \xra{\sim} Q_m$), it suffices to obtain an equivalence
\[ |\Gr^{-1}(P_{m,\bullet})| \simeq \pt_{\Fun(Q_m,\S)} . \]
As colimits in $\Fun(Q_m,\S)$ are computed pointwise, for this it suffices to show that for any point $q \in Q_m$, we have
\[ | \Gr^{-1}(P_{m,\bullet})(q) | \simeq \pt_\S . \]
Moreover, the naturality of the Grothendieck construction implies that we can identify the constituent simplicial spaces of this geometric realization as
\[ \Gr^{-1}(P_{m,n})(q) \simeq \lim \left( \begin{tikzcd}
& P_{m,n} \arrow{d} \\
\pt_\S \arrow{r}[swap]{q} & Q_m
\end{tikzcd} \right) \]
for all $n \geq 0$ in a way compatible with the simplicial structure maps; in other words, we have an equivalence
\[ \Gr^{-1}(P_{m,\bullet})(q) \simeq \lim \left( \begin{tikzcd}[column sep=1.5cm]
& P_{m,\bullet} \arrow{d} \\
\pt_{s\S} \arrow{r}[swap]{\const(q)} & \const(Q_m)
\end{tikzcd} \right) \]
in $s\S$.

Now, by definition $Q_m = \hom_\Modelip(\word{q}^m,\M)$, and so our point $q \in Q_m$ corresponds to some map $\word{q}^m \xra{q'} \M$ in $\Modelip$.  Via this map we can consider $\M \in ( \Modelip)_{\word{q}^m/}$, and it is not hard to see that we have equivalences
\[ \lim \left( \begin{tikzcd}[column sep=1.5cm]
& P_{m,\bullet} \arrow{d} \\
\pt_{s\S} \arrow{r}[swap]{\const(q)} & \const(Q_m)
\end{tikzcd} \right) \simeq \hom^\lw_{(\Modelip)_{\word{q}^m/}}(\word{p}^{m,\bullet},\M) \simeq \Nervei \left( \left( \bW_{y \dcofibn} \right)_{(y \wcofibn q'(\beta(i)))/} \right) . \]
But this last simplicial space is the nerve of an $\infty$-category with an initial object, so it has contractible geometric realization by \cref{rnerves:groupoid-completion of CSSs} and the opposite of \cref{gr:groupoid-completion of infty-cat with terminal object is contractible}.  Thus, we have shown that $|P_{m,\bullet}| \xra{\sim} Q_m$, which as we have seen implies that $|P_\bullet| \xra{\sim} |Q_\bullet|$, i.e.\! that $P_\bullet \ra Q_\bullet$ lies in $\bW_\KQ$.

\item The maps $P_\bullet \ra P_\bullet \times \Delta^1$ are given by
\[ P_\bullet \simeq P_\bullet \times \Delta^{\{i\}} \ra P_\bullet \times \Delta^1 , \]
where we take $i=0$ for the horizontal map and $i=1$ for the vertical map.  These lie in $\bW_\KQ$ since the geometric realization functor $|{-}|:s\S \ra \S$ (as a sifted colimit) commutes with finite products.

\item The map $P_\bullet \times \Delta^1 \ra Q_\bullet$ is the corepresented left homotopy associated to the left homotopy corepresentation
\[ \{ \{ _\word{q}h^i_n \in \hom_\Modelp(\word{q}^{n+1} , \word{p}^n) \}_{0 \leq i \leq n} \}_{n \geq 0} \]
given by defining $_\word{q}h^i_n$ to be the unique functor satisfying
\begin{align*}
\alpha(j) & \mapsto \left\{ \begin{array}{ll} 
\alpha(j) , & j \leq i \\
\alpha(j-1), & j > i
\end{array} \right. \\
\beta(j) & \mapsto \left\{ \begin{array}{ll}
\beta(j), & j \leq i \\
\gamma(j), & j > i .
\end{array} \right.
\end{align*}
(It is tedious but straightforward to verify that these formulas do indeed define a suitable left homotopy corepresentation.)

\end{itemize}
Thus, we have exhibited the above original commutative diagram in $s\S$ and shown that it lies entirely in $\bW_\KQ$.  In particular, it follows that $|M_\bullet| \xra{\sim} |Q_\bullet|$, i.e.\! that
\[ \left\| \hom^\lw_\M(\scyl^\bullet(x),\spth_\bullet(y)) \right\| \xra{\sim} \tilde{\word{3}}(x,y)^\gpd ,\]
as desired.
\end{proof}

We now prove an auxiliary result which was needed in the proof of \cref{bisimplicial colimit and special-3}, an analog of \cite[Proposition 6.11]{DKFunc}.\footnote{The proof of \cite[Proposition 6.11]{DKFunc} contains a mild but rather confusing typo.  There, it is claimed that a certain category is isomorphic to the homotopy colimit of a simplicial set, which is then claimed to have the same homotopy type as another simplicial set.  In fact, it is the \textit{nerve} of the category which is \textit{isomorphic} to the first simplicial set itself (without saying ``homotopy colimit''), and then this simplicial set is equivalent to the other simplicial set because the latter is the nerve of the category of simplices of the former.  This last statement can be seen as coming from the fact that there are two ways to take the homotopy colimit of a simplicial set: either by taking its usual geometric realization, or by taking the geometric realization of its simplicial replacement.}

\begin{lem}\label{path objects are final}
If $y \in \M^f$ is fibrant and $\spth_\bullet(y) \in s\M$ is any special path object for $y$, then the functor
\begin{align*}
\bD^{op} & \ra \bW_{y \dcofibn} \\
[n]^\opobj & \mapsto (y \wcofibn \spth_n(y))
\end{align*}
is final.
\end{lem}

\begin{proof}
According to the characterization of Theorem A (\Cref{gr:theorem A}), it suffices to show that for any object $(y \wcofibn z) \in \bW_{y \dcofibn}$, the groupoid completion of the comma $\infty$-category
\[ \bD^{op} \underset{\bW_{y \dcofibn}}{\times} \left( \bW_{y \dcofibn} \right)_{(y \wcofibn z) / } \]
is contractible.

First of all, note that the chosen equivalence $y \simeq \spth_0(y)$ endows the object $\hom_\M^\lw(y,\spth_\bullet(y)) \in s\S$ with a canonical basepoint $\pt_{s\S} \ra \hom^\lw_\M(y,\spth_\bullet(y))$.  The dual of \cref{homotopy classes of maps and bisimplicial colimit}\ref{preserve wfibns} implies that the map
\[ \hom_\M^\lw(z,\spth_\bullet(y)) \ra \hom_\M^\lw(y,\spth_\bullet(y)) \]
is in $(\bW \cap \bF)_\KQ$, which implies (by \cref{sspaces:the fiber of a fibration of sspaces is homotopically correct}) that its fiber over that basepoint has contractible geometric realization.  As fibers (being limits) in $s\S = \Fun(\bD^{op},\S)$ are computed objectwise, this fiber is given in level $n$ by
\[ \hom_{ \left( \bW_{y \dcofibn} \right) } \left( y \wcofibn z, y \wcofibn \spth_n(y) \right) . \]
(Note that the inclusions $\bW_{y \dcofibn} \subset \bW_{y/} \subset \M_{y/}$ are both inclusions of full subcategories (the latter by the two-out-of-three property).)  By the Bousfield--Kan colimit formula (\cref{gr:bousfield--kan}), the geometric realization of this simplicial space is equivalent to the geometric realization of its simplicial replacement when considered in $s\S = \Fun(\bD^{op},\S)$.  In level $n$, this simplicial replacement is given by
\[ \coprod_{\alpha \in \Nerve(\bD^{op})_n} \hom_{\left(\bW_{y \dcofibn}\right)} \left( y \wcofibn z , y \wcofibn \spth_{\alpha(0)}(y) \right) . \]

We claim that this latter simplicial space is precisely the nerve of the comma $\infty$-category
\[ \bD^{op} \underset{\bW_{y\dcofibn}}{\times} (\bW_{y\dcofibn})_{(y \wcofibn z) / } . \]
To see this, observe that
\begin{align*}
\Nervei \left(  \bD^{op} \underset{\bW_{y\dcofibn}}{\times} (\bW_{y\dcofibn})_{(y \wcofibn z) / }  \right)_n &= \hom_\Cati \left( [n] , \bD^{op} \underset{\bW_{y \dcofibn}}{\times} {\left( \bW_{y \dcofibn} \right)}_{(y \wcofibn z)/} \right) \\
& \simeq \lim \left( \begin{tikzcd}[ampersand replacement=\&]
\& \hom_\Cati \left( [n] , \left( \bW_{y\dcofibn} \right)_{(y \wcofibn z)/} \right) \arrow{d} \\
\hom_\Cati ([n] , \bD^{op}) \arrow{r} \& \hom_\Cati \left([n],\bW_{y \dcofibn} \right)
\end{tikzcd} \right).
\end{align*}
Since $\hom_\Cati ([n],\bD^{op}) \simeq \Nerve(\bD^{op})_n$ is discrete, this pullback is equivalent to a coproduct over its elements of the corresponding fibers.  Over the element $\alpha \in \Nerve(\bD^{op})_n$, this fiber is
\begin{align*}
& \lim \left( \begin{tikzcd}[ampersand replacement=\&]
\& \hom_\Cati \left( [n] , \left( \bW_{y\dcofibn} \right)_{(y \wcofibn z)/} \right) \arrow{d} \\
\left\{ [n] \xra{\alpha} \bD^{op} \ra \bW_{y \dcofibn} \right\} \arrow[hook]{r} \& \hom_\Cati \left([n],\bW_{y \dcofibn} \right)
\end{tikzcd} \right)  \\
& \simeq \lim \left( \begin{tikzcd}[ampersand replacement=\&]
\& \& \{ (y \wcofibn z) \} \\
\& \hom_\Cati \left( \{(-1) \ra \cdots \ra n\} , \bW_{y \dcofibn} \right) \arrow{d} \arrow{r} \&  \hom_\Cati \left( \{(-1)\} , \bW_{y \dcofibn} \right)  \arrow[hookleftarrow]{u} \\
\left\{ [n] \xra{\alpha} \bD^{op} \ra \bW_{y \dcofibn} \right\} \arrow[hook]{r} \& \hom_\Cati \left([n],\bW_{y \dcofibn} \right)
\end{tikzcd} \right) \\
& \simeq \lim \left( \begin{tikzcd}[ampersand replacement=\&]
\& \& \{ (y \wcofibn z) \} \\
\& \hom_\Cati \left( \{(-1) \ra 0\} , \bW_{y \dcofibn} \right) \arrow{d} \arrow{r} \&  \hom_\Cati \left( \{(-1)\} , \bW_{y \dcofibn} \right) \arrow[hookleftarrow]{u} \\
\left\{ (y \wcofibn \spth_{\alpha(0)}(y)) \right\} \arrow[hook]{r} \& \hom_\Cati \left([0],\bW_{y \dcofibn} \right)
\end{tikzcd} \right) \\
& \simeq \hom_{\left( \bW_{y \dcofibn} \right)} \left( y \wcofibn z , y \wcofibn \spth_{\alpha(0)}(y) \right) .
\end{align*}
Moreover, it is clear that the structure maps of this simplicial space agree with those of the above simplicial replacement: both are ultimately induced by the structure maps of $\spth_\bullet(y) \in s\M$.  So, these are indeed equivalent simplicial spaces.

We have just shown that the geometric realization of the complete Segal space
\[ \Nervei \left( \bD^{op} \underset{\bW_{y\dcofibn}}{\times} (\bW_{y\dcofibn})_{(y \wcofibn z) / } \right) \]
is contractible.  Thus, by \cref{rnerves:groupoid-completion of CSSs}, the groupoid completion
\[ \left( \bD^{op} \underset{\bW_{d\dcofibn}}{\times} (\bW_{d\dcofibn})_{(d \wcofibn d') / } \right)^{\gpd} \]
is indeed contractible.
\end{proof}

\section{The equivalence $\tilde{\word{3}}(x,y)^\gpd \simeq \word{3}(x,y)^\gpd$}\label{section special-3 to 3}

We now prove that the $\infty$-category of three-arrow zigzags from $x$ to $y$ has equivalent groupoid completion to that of its subcategory of special three-arrow zigzags.

\begin{prop}\label{connect special-3 to 3}
For any model $\infty$-category $\M$ and any $x,y \in \M$, the unique map $\word{3} \ra \tilde{\word{3}}$ in $\Modelp$ induces an equivalence
\[ \tilde{\word{3}}(x,y)^\gpd \xra{\sim} \word{3}(x,y)^\gpd \]
on groupoid completions of $\infty$-categories of zigzags in $\M$ from $x$ to $y$.
\end{prop}

\begin{proof}
We apply the functor $\left( \Funpdec(-,\M)^\bW \right)^\gpd$ to the sequence of maps in $\Modelpdec$ given in the proof of \cite[\cref{adjns:special-3 to 3}\cref{adjns:special-3 to 3 for model cat}]{adjns} (which factors the unique map $\word{3} \ra \tilde{\word{3}}$ in $\Modelp$).  To show that the induced maps in $\S$ are all equivalences, the arguments given there generalize as follows.
\begin{itemize}
\item To show that the maps $\varphi_1$ and $\varphi_4$ defined there induce equivalences in $\S$, we replace the appeal to \cite[\cref{adjns:abstractify DKFunc 8.1}\cref{adjns:abstractify for model cat}]{adjns} with an appeal to the factorization lemma (\ref{factorization lemma}).
\item The maps $\varphi_2$ and $\varphi_5$ defined there even induce equivalences in $\Cati$ upon application of $\Funpdec(-,\M)^\bW$; to see this, we use the argument given in the proof of \cref{connect 3 to 7} for why the maps $\varphi_2$, $\varphi_4$, $\varphi_9$, and $\varphi_{11}$ (of that proof) have this same property.
\item To show that the maps $\varphi_3$ and $\varphi_6$ defined there induce equivalences in $\S$, we use the argument given in the proof of \cref{connect 3 to 7} for why the maps $\varphi_7$ and $\varphi_{14}$ (of that proof) have this same property.
\end{itemize}
Thus, we obtain the desired equivalence $\tilde{\word{3}}(x,y)^\gpd \simeq \word{3}(x,y)^\gpd$ in $\S$.
\end{proof}

\section{The equivalence $\word{3}(x,y)^\gpd \simeq \word{7}(x,y)^\gpd$}\label{section 3 to 7}

We now prove that the $\infty$-categories of three-arrow zigzags and seven-arrow zigzags from $x$ to $y$ have equivalent groupoid completions.

\begin{prop}\label{connect 3 to 7}
If $\M$ is a model $\infty$-category, then for any $x,y \in \M$, the map $\word{7} \ra \word{3}$ in $\Modelp$ given by collapsing the middle four instances of $\bW^\pm$ induces an equivalence
\[ \word{3}(x,y)^\gpd \xra{\sim} \word{7}(x,y)^\gpd \]
on groupoid completions of $\infty$-categories of zigzags in $\M$ from $x$ to $y$.
\end{prop}

\begin{proof}
In essence, we use the factorization lemma (\ref{factorization lemma}) to remove each instance of $\bW^{-1}$ in $\word{7}$ which is adjacent to the unique instance of $\any$, and then we ``compose out'' the remaining instances of $\bW$.  To be precise, we define a diagram
\[ \word{7} \xra{\varphi_1} \I_1 \xra{\varphi_2} \I_2 \xla{\varphi_3} \I_3 \xra{\varphi_4} \I_4 \xla{\varphi_5} \I_5 \xla{\varphi_6} \I_6 \xla{\varphi_7} \I_7 \xra{\varphi_8} \I_8 \xra{\varphi_9} \I_9 \xla{\varphi_{10}} \I_{10} \xra{\varphi_{11}} \I_{11} \xra{\varphi_{12}} \I_{12} \xla{\varphi_{13}} \I_{13} \xla{\varphi_{14}} \word{3} \]
in $\Modelpdec$, given by
\begin{align*}
\word{7}
& = \left( \begin{tikzcd}[ampersand replacement=\&, column sep=0.5cm]
s \& \& \bullet \arrow{ll}[swap]{\approx} \arrow{rr}{\approx} \& \& \bullet \& \& \bullet \arrow{ll}[swap]{\approx} \arrow{rr} \& \& \bullet \& \& \bullet \arrow{ll}[swap]{\approx} \arrow{rr}{\approx} \& \& \bullet \& \& t \arrow{ll}[swap]{\approx}
\end{tikzcd} \right)
\\
& \xra{\varphi_1} \left( \begin{tikzcd}[ampersand replacement=\&, column sep=0.5cm]
s \& \& \bullet \arrow{ll}[swap]{\approx} \arrow{rr}{\approx} \& \& \bullet \& \& \bullet \arrow{ll}[swap]{\approx} \arrow{rr} \arrow[tail]{ld}[sloped, pos=0.7]{\approx} \& \& \bullet \& \& \bullet \arrow{ll}[swap]{\approx} \arrow{rr}{\approx} \& \& \bullet \& \& t \arrow{ll}[swap]{\approx} \\
\& \& \& \& \& \bullet \arrow[two heads]{lu}[sloped, pos=0.3]{\approx}
\end{tikzcd} \right)
\\
& \xra{\varphi_2} \left( \begin{tikzcd}[ampersand replacement=\&, column sep=0.5cm]
s \& \& \bullet \arrow{ll}[swap]{\approx} \arrow{rr}{\approx} \& \& \bullet \& \& \bullet \arrow{ll}[swap]{\approx} \arrow{rr} \arrow[tail]{ld}[sloped, pos=0.7]{\approx} \& \& \bullet \& \& \bullet \arrow{ll}[swap]{\approx} \arrow{rr}{\approx} \& \& \bullet \& \& t \arrow{ll}[swap]{\approx} \\
\& \& \& \bullet \arrow{rr}[swap]{\approx} \arrow[two heads]{lu}[sloped, pos=0.3]{\approx} \& \& \bullet \arrow[two heads]{lu}[sloped, pos=0.3]{\approx}
\end{tikzcd} \right)^{\textup{p.b.}}
\\
& \xla{\varphi_3} \left( \begin{tikzcd}[ampersand replacement=\&, column sep=0.5cm]
s \& \& \bullet \arrow{ll}[swap]{\approx} \arrow{rr}{\approx} \& \& \bullet \& \& \bullet \arrow{ll}[swap]{\approx} \arrow{rr} \arrow[tail]{ld}[sloped, pos=0.7]{\approx} \& \& \bullet \& \& \bullet \arrow{ll}[swap]{\approx} \arrow{rr}{\approx} \& \& \bullet \& \& t \arrow{ll}[swap]{\approx} \\
\& \& \& \bullet \arrow{rr}[swap]{\approx} \arrow[two heads]{lu}[sloped, pos=0.3]{\approx} \& \& \bullet \arrow[two heads]{lu}[sloped, pos=0.3]{\approx}
\end{tikzcd} \right)
\\
& \xra{\varphi_4} \left( \begin{tikzcd}[ampersand replacement=\&, column sep=0.5cm]
s \& \& \bullet \arrow{ll}[swap]{\approx} \arrow{rr}{\approx} \& \& \bullet \& \& \bullet \arrow{ll}[swap]{\approx} \arrow{rr} \arrow[tail]{ld}[sloped, pos=0.7]{\approx} \& \& \bullet \arrow[tail]{ld}[sloped, pos=0.7]{\approx} \& \& \bullet \arrow{ll}[swap]{\approx} \arrow{rr}{\approx} \& \& \bullet \& \& t \arrow{ll}[swap]{\approx} \\
\& \& \& \bullet \arrow{rr}[swap]{\approx} \arrow[two heads]{lu}[sloped, pos=0.3]{\approx} \& \& \bullet \arrow[two heads]{lu}[sloped, pos=0.3]{\approx} \arrow{rr} \& \& \bullet
\end{tikzcd} \right)^{\textup{p.o.}}
\\
& \xla{\varphi_5} \left( \begin{tikzcd}[ampersand replacement=\&, column sep=0.5cm]
s \& \& \bullet \arrow{ll}[swap]{\approx} \arrow{rr}{\approx} \& \& \bullet \& \& \bullet \arrow{ll}[swap]{\approx} \arrow{rr} \arrow[tail]{ld}[sloped, pos=0.7]{\approx} \& \& \bullet \arrow[tail]{ld}[sloped, pos=0.7]{\approx} \& \& \bullet \arrow{ll}[swap]{\approx} \arrow{rr}{\approx} \& \& \bullet \& \& t \arrow{ll}[swap]{\approx} \\
\& \& \& \bullet \arrow{rr}[swap]{\approx} \arrow[two heads]{lu}[sloped, pos=0.3]{\approx} \& \& \bullet \arrow[two heads]{lu}[sloped, pos=0.3]{\approx} \arrow{rr} \& \& \bullet
\end{tikzcd} \right)
\\
& \xla{\varphi_6} \left( \begin{tikzcd}[ampersand replacement=\&, column sep=0.5cm]
s \& \& \bullet \arrow{ll}[swap]{\approx} \arrow{rr}{\approx} \& \& \bullet \arrow{rr} \& \& \bullet \& \& \bullet \arrow{ll}[swap]{\approx} \arrow{rr}{\approx} \& \& \bullet \& \& t \arrow{ll}[swap]{\approx}
\end{tikzcd} \right)
\\
& \xla{\varphi_7} \left( \begin{tikzcd}[ampersand replacement=\&, column sep=0.5cm]
s \& \& \bullet \arrow{ll}[swap]{\approx} \arrow{rr} \& \& \bullet \& \& \bullet \arrow{ll}[swap]{\approx} \arrow{rr}{\approx} \& \& \bullet \& \& t \arrow{ll}[swap]{\approx}
\end{tikzcd} \right)
\\
& \xra{\varphi_8} \left( \begin{tikzcd}[ampersand replacement=\&, column sep=0.5cm]
s \& \& \bullet \arrow{ll}[swap]{\approx} \arrow{rr} \& \& \bullet \& \& \bullet \arrow{ll}[swap]{\approx} \arrow{rr}{\approx} \arrow[tail]{ld}[sloped, pos=0.7]{\approx} \& \& \bullet \& \& t \arrow{ll}[swap]{\approx} \\
\& \& \& \& \& \bullet \arrow[two heads]{lu}[sloped, pos=0.3]{\approx}
\end{tikzcd} \right)
\\
& \xra{\varphi_9} \left( \begin{tikzcd}[ampersand replacement=\&, column sep=0.5cm]
s \& \& \bullet \arrow{ll}[swap]{\approx} \arrow{rr} \& \& \bullet \& \& \bullet \arrow{ll}[swap]{\approx} \arrow{rr}{\approx} \arrow[tail]{ld}[sloped, pos=0.7]{\approx} \& \& \bullet \& \& t \arrow{ll}[swap]{\approx} \\
\& \& \& \bullet \arrow[two heads]{lu}[sloped, pos=0.3]{\approx} \arrow{rr} \& \& \bullet \arrow[two heads]{lu}[sloped, pos=0.3]{\approx}
\end{tikzcd} \right)^{\textup{p.b.}}
\\
& \xla{\varphi_{10}} \left( \begin{tikzcd}[ampersand replacement=\&, column sep=0.5cm]
s \& \& \bullet \arrow{ll}[swap]{\approx} \arrow{rr} \& \& \bullet \& \& \bullet \arrow{ll}[swap]{\approx} \arrow{rr}{\approx} \arrow[tail]{ld}[sloped, pos=0.7]{\approx} \& \& \bullet \& \& t \arrow{ll}[swap]{\approx} \\
\& \& \& \bullet \arrow[two heads]{lu}[sloped, pos=0.3]{\approx} \arrow{rr} \& \& \bullet \arrow[two heads]{lu}[sloped, pos=0.3]{\approx}
\end{tikzcd} \right)
\\
& \xra{\varphi_{11}} \left( \begin{tikzcd}[ampersand replacement=\&, column sep=0.5cm]
s \& \& \bullet \arrow{ll}[swap]{\approx} \arrow{rr} \& \& \bullet \& \& \bullet \arrow{ll}[swap]{\approx} \arrow{rr}{\approx} \arrow[tail]{ld}[sloped, pos=0.7]{\approx} \& \& \bullet \arrow[tail]{ld}[sloped, pos=0.7]{\approx} \& \& t \arrow{ll}[swap]{\approx} \\
\& \& \& \bullet \arrow[two heads]{lu}[sloped, pos=0.3]{\approx} \arrow{rr} \& \& \bullet \arrow[two heads]{lu}[sloped, pos=0.3]{\approx} \arrow{rr}[swap]{\approx} \& \& \bullet
\end{tikzcd} \right)^{\textup{p.o.}}
\\
& \xla{\varphi_{12}} \left( \begin{tikzcd}[ampersand replacement=\&, column sep=0.5cm]
s \& \& \bullet \arrow{ll}[swap]{\approx} \arrow{rr} \& \& \bullet \& \& \bullet \arrow{ll}[swap]{\approx} \arrow{rr}{\approx} \arrow[tail]{ld}[sloped, pos=0.7]{\approx} \& \& \bullet \arrow[tail]{ld}[sloped, pos=0.7]{\approx} \& \& t \arrow{ll}[swap]{\approx} \\
\& \& \& \bullet \arrow[two heads]{lu}[sloped, pos=0.3]{\approx} \arrow{rr} \& \& \bullet \arrow[two heads]{lu}[sloped, pos=0.3]{\approx} \arrow{rr}[swap]{\approx} \& \& \bullet
\end{tikzcd} \right)
\\
& \xla{\varphi_{13}} \left( \begin{tikzcd}[ampersand replacement=\&, column sep=0.5cm]
s \& \& \bullet \arrow{ll}[swap]{\approx} \arrow{rr} \& \& \bullet \arrow{rr}{\approx} \& \& \bullet \& \& t \arrow{ll}[swap]{\approx}
\end{tikzcd} \right)
\\
& \xla{\varphi_{14}} \left( \begin{tikzcd}[ampersand replacement=\&, column sep=0.5cm]
s \& \& \bullet \arrow{ll}[swap]{\approx} \arrow{rr} \& \& \bullet \& \& t \arrow{ll}[swap]{\approx}
\end{tikzcd} \right) = \word{3} ,
\end{align*}
where all maps are the completely evident inclusions, except that
\begin{itemize}
\item $\varphi_6$ and $\varphi_{13}$ are the ``lower inclusions'' (whose images omit any objects in the upper rows that are the source or target of a drawn-in diagonal arrow -- note that there are certain ``hidden'' diagonal maps in $\I_5$ and $\I_{12}$, which are only composites of drawn-in arrows), and
\item $\varphi_7$ and $\varphi_{14}$ are obtained by taking the unique copy of $\any$ onto the composite $[\bW ; \any]$ or $[\any;\bW]$, respectively.
\end{itemize}
We claim that this induces a diagram of equivalences in $\S$ upon application of $\left( \Funpdec(-,\M)^\bW \right)^\gpd$.  The arguments can be grouped as follows.
\begin{itemize}

\item The maps $\varphi_1$ and $\varphi_8$ induce equivalences in $\S$ by the factorization lemma (\ref{factorization lemma}).

\item The maps $\varphi_2$, $\varphi_4$, $\varphi_9$, and $\varphi_{11}$ actually even induce equivalences in $\Cati$ upon application of $\Funpdec(-,\M)^\bW$; this follows from the facts that

\begin{itemize}

\item $\M$ is finitely bicomplete,

\item the subcategories $(\bW \cap \bF) , (\bW \cap \bC) \subset \M$ are respectively closed under pullbacks and pushouts, and

\item the subcategory $\bW \subset \M$ has the two-out-of-three property

\end{itemize}
(see e.g.\! Proposition T.4.3.2.15).

\item Upon application of $\Funpdec(-,\M)^\bW$, the maps $\varphi_3$ and $\varphi_{10}$ induce functors which admit left adjoints, and so they induce equivalences in $\S$ upon application of $\left(\Funpdec(-,\M)^\bW\right)^\gpd$ by \cref{rnerves:adjns induce equivces on gpd-complns}.  Dually, the maps $\varphi_5$ and $\varphi_{12}$ also induce equivalences in $\S$.

\item The maps $\varphi_6$, $\varphi_7$, $\varphi_{13}$, and $\varphi_{14}$ admit evident retractions $\psi_6$, $\psi_7$, $\psi_{13}$, and $\psi_{14}$, respectively.  Moreover,

\begin{itemize}

\item there are evident cospans of doubly-pointed natural weak equivalences connecting $\id_{\I_5}$ with $\varphi_6 \circ \psi_6$ and connecting $\id_{\I_{12}}$ with $\varphi_{13} \circ \psi_{13}$, and

\item there are evident doubly-pointed natural weak equivalences $\varphi_7 \circ \psi_7 \we \id_{\I_6}$ and $\id_{\I_{13}} \we \varphi_{14} \circ \psi_{14}$.

\end{itemize}
Hence, by Lemmas \ref{nat w.e. induces nat trans for model diagrams} \and \Cref{rnerves:nat trans induces equivce betw maps on gpd-complns}, these maps all induce equivalences in $\S$.

\end{itemize}
Thus, we obtain the desired equivalence $\word{3}(x,y)^\gpd \simeq \word{7}(x,y)^\gpd$ in $\S$ which, tracing back through the above zigzag in $\Modelpdec$, it is clear is indeed induced by the asserted map $\word{7} \ra \word{3}$ in $\Modelp$.
\end{proof}

\section{Localization of model $\infty$-categories}\label{section model infty-cat gives Ss}

So far, given a model $\infty$-category $\M$ and suitably co/fibrant objects $x,y \in \M$, we have related the spaces of left/right homotopy classes of maps from $x$ to $y$ to the groupoid completions of various $\infty$-categories of zigzags from $x$ to $y$.  However, in order to show that these are all actually equivalent to the space $\hom_{\loc{\M}{\bW}}(x,y)$ of maps from $x$ to $y$ in the localization $\loc{\M}{\bW}$, we must access this latter hom-space.  This aim is one of the primary purposes of the local universal property of the Rezk nerve (\cref{rnerves:rezk nerve of a relative infty-category is initial}) and the calculus theorem (\Cref{hammocks:calculus result}), which we now bring to fruition.  The following result will be strictly generalized by \cref{rnerve is a CSS}, but the latter actually requires the full force of the fundamental theorem of $\infty$-categories (\cref{fundamental theorem}).  Thus, to avoid circularity, we prove only this weaker version first.

\begin{prop}\label{rnerve is a SS}
If $\M$ is a model $\infty$-category with underlying relative $\infty$-category $(\M,\bW)$, then $\NerveRezki(\M,\bW) \in \SS$, and moreover the morphism $\Nervei(\M) \ra \leftloc_\CSS(\NerveRezki(\M,\bW))$ in $\CSS$ corresponds to the morphism $\M \ra \loc{\M}{\bW}$ in $\Cati$.
\end{prop}

\begin{proof}
The first claim is obtained by combining \cref{model infty-cats have calculi} and the calculus theorem (\Cref{hammocks:calculus result}\cref{hammocks:calculus for SS}), while the second claim follows from the local universal property of the Rezk nerve (\cref{rnerves:rezk nerve of a relative infty-category is initial}).
\end{proof}

We now give an auxiliary result on which the proof of \cref{rnerve is a SS} relies.

\begin{lem}\label{model infty-cats have calculi}
If $\M$ is a model $\infty$-category, then its underlying relative $\infty$-category $(\M,\bW)$ admits a homotopical three-arrow calculus.
\end{lem}

\begin{proof}
After choosing any pair of objects $x,y \in \M$, we apply the functor $\left( \Funpdec(-,\M)^\bW \right)^\gpd$ to the diagram in $\Modelpdec$ given in the proof of \cite[\cref{adjns:the two cases have homotopical three-arrow calculi}\cref{adjns:model cat has homotopical three-arrow calculus}]{adjns}.  To show that the induced maps in $\S$ are all equivalences, the arguments given there generalize as follows.
\begin{itemize}
\item To show that the map $\rho_1$ defined there induces an equivalence in $\S$, we replace the appeal to \cite[\cref{adjns:abstractify DKFunc 8.1}\cref{adjns:abstractify for model cat}]{adjns} with an appeal to the factorization lemma (\ref{factorization lemma}).
\item The map $\rho_2$ defined there even induces an equivalence in $\Cati$ upon application of $\Funpdec(-,\M)^\bW$; to see this, we repeatedly apply the argument given in the proof of \cref{connect 3 to 7} for why the maps $\varphi_2$, $\varphi_4$, $\varphi_9$, and $\varphi_{11}$ (of that proof) have this same property.
\item The map $\rho_3$ defined there induces an equivalence in $\S$ in exactly the same manner; we replace the appeal to \cite[\cref{adjns:natural weak equivalences of model diagrams corepresent natural transformations}]{adjns} with an appeal to Lemmas \ref{nat w.e. induces nat trans for model diagrams} \and \Cref{rnerves:nat trans induces equivce betw maps on gpd-complns}.
\end{itemize}
Thus, the underlying relative $\infty$-category $(\M,\bW)$ of the model $\infty$-category $\M$ does indeed admit a homotopical three-arrow calculus.
\end{proof}

\section{The equivalence $\word{7}(x,y)^\gpd \simeq \hom_{\loc{\M}{\bW}}(x,y)$}\label{section equivalence of 7 with hom in loc}

In this section, we show that the groupoid completion of the $\infty$-category of seven-arrow zigzags from $x$ to $y$ is equivalent to the hom-space $\hom_{\loc{\M}{\bW}}(x,y)$, thus completing the string of equivalences in the proof of the fundamental theorem of model $\infty$-categories (\ref{fundamental theorem}).


\begin{prop}\label{connect hom in loc to 7}
For any model $\infty$-category $\M$ and any $x,y \in \M$, we have a canonical equivalence
\[ \word{7}(x,y)^\gpd \xra{\sim} \hom_{\loc{\M}{\bW}}(x,y) . \]
\end{prop}

\begin{proof}
First of all, by \cref{rnerve is a SS} (and \cref{hammocks:extract hom-spaces from a SS}), we have
\begin{align*}
\hom_{\loc{\M}{\bW}}(x,y)
& \simeq \lim \left( \begin{tikzcd}[ampersand replacement=\&]
\& \NerveRezki(\M,\bW)_1 \arrow{d}{(s,t)} \\
\pt_\S \arrow{r}[swap]{(x,y)} \& \NerveRezki(\M,\bW)_0 \times \NerveRezki(\M,\bW)_0
\end{tikzcd} \right) \\
& \simeq \lim \left( \begin{tikzcd}[ampersand replacement=\&]
\& \& \pt_\S \arrow{d}{y} \\
\& \NerveRezki(\M,\bW)_1 \arrow{r}[swap]{t} \arrow{d}{s} \& \NerveRezki(\M,\bW)_0 \\
\pt_\S \arrow{r}[swap]{x} \& \NerveRezki(\M,\bW)_0
\end{tikzcd} \right) \\
& = \lim \left( \begin{tikzcd}[ampersand replacement=\&]
\& \& \pt_\S \arrow{d}{y} \\
\& (\Fun([1],\M)^\bW)^\gpd \arrow{r}[swap]{t^\gpd} \arrow{d}{s^\gpd} \& (\Fun([0],\M)^\bW)^\gpd \\
\pt_\S \arrow{r}[swap]{x} \& (\Fun([0],\M)^\bW)^\gpd
\end{tikzcd} \right) \\
& \simeq \lim \left( \begin{tikzcd}[ampersand replacement=\&]
\& \& (\pt_\Cati)^\gpd \arrow{d}{y^\gpd} \\
\& (\Fun([1],\M)^\bW)^\gpd \arrow{r}[swap]{t^\gpd} \arrow{d}{s^\gpd} \& \bW^\gpd \\
(\pt_\Cati)^\gpd \arrow{r}[swap]{x^\gpd} \& \bW^\gpd
\end{tikzcd} \right) .
\end{align*}
Note that this final limit is that of a diagram in $\S$ coming from a diagram in $\Cati$ via postcomposition with $(-)^\gpd : \Cati \ra \S$.  We will compute this limit by first computing the pullback of the lower left cospan (defined by the maps $x$ and $s$) and then computing the pullback of the resulting cospan; for both pullbacks we will appeal to Theorems {\Bn} and {\Cn} (\Cref{gr:thm Bn} and \Cref{gr:thm Cn}), noting once and for all that $\bW^{op}$ has property {\Cthree} by Lemmas \ref{htpical 3-arrow calc implies W-op has C3} \and \ref{model infty-cats have calculi}.

First of all, by Theorem {\Cn} (\Cref{gr:thm Cn}), the functor
\[ (\pt_\Cati)^{op} \xra{x^\opobj} \bW^{op} \]
has property {\Bthree}.  Hence, by Theorem {\Bn} (\Cref{gr:thm Bn}), we have a homotopy pullback square
\[ \begin{tikzcd}
( x^\opobj((\pt_\Cati)^{op}) \da_3 s^{op}((\Fun([1],\M)^\bW)^{op})) \arrow{r}{t} \arrow{d}[swap]{s} & (\Fun([1],\M)^\bW)^{op} \arrow{d}{s^{op}} \\
(\pt_\Cati)^{op} \arrow{r}[swap]{x^\opobj} & \bW^{op}
\end{tikzcd} \]
in $(\Cati)_\Thomason$; unwinding the definitions, we can identify the homotopy pullback as
\[ (\Fun_{*\circ}([\bW^{-1};\bW;\bW^{-1};\any],\M)^\bW)^{op} , \]
where the object $x \in \M$ determines the pointing.  As homotopy pullback squares in $(\Cati)_\Thomason$ are preserved under the involution $(-)^{op} : \Cati \ra \Cati$, it follows that we have a pullback square
\[ \begin{tikzcd}
(\Fun_{*\circ}([\bW^{-1};\bW;\bW^{-1};\any],\M)^\bW)^\gpd \arrow{r} \arrow{d} & (\Fun([1],\M)^\bW)^\gpd \arrow{d}{s^\gpd} \\
(\pt_\Cati)^\gpd \arrow{r}[swap]{x^\gpd} & \bW^\gpd
\end{tikzcd} \]
in $\S$, and hence we can simplify the above limit computing $\hom_{\loc{\M}{\bW}}(x,y)$ to give the identification
\[ \hom_{\loc{\M}{\bW}}(x,y) \simeq \lim \left( \begin{tikzcd}
& (\pt_\Cati)^\gpd \arrow{d}{y} \\
(\Fun_{*\circ}([\bW^{-1};\bW;\bW^{-1};\any],\M)^\bW)^\gpd \arrow{r}[swap]{t^\gpd} & \bW^\gpd
\end{tikzcd} \right) . \]

Then, again by Theorem {\Cn} (\Cref{gr:thm Cn}), the functor
\[ (\Fun_{*\circ}([\bW^{-1};\bW;\bW^{-1};\any],\M)^\bW)^{op} \xra{t^{op}} \bW^{op}  \]
has property {\Bthree}, so that by Theorem {\Bn} (\Cref{gr:thm Bn}) we have a homotopy pullback square
\[ \begin{tikzcd}
( t^{op}((\Fun_{*\circ}([\bW^{-1};\bW;\bW^{-1};\any],\M)^\bW)^{op}) \da_3 y^\opobj((\pt_\Cati)^{op})) \arrow{r}{t} \arrow{d}[swap]{s} & (\pt_\Cati)^{op} \arrow{d}{y^\opobj} \\
(\Fun_{*\circ}([\bW^{-1};\bW;\bW^{-1};\any],\M)^\bW)^{op} \arrow{r}[swap]{t^{op}} & \bW^{op}
\end{tikzcd} \]
in $(\Cati)_\Thomason$; this time, unwinding the definitions we can identify the homotopy pullback as
\[ (\Funp([\bW^{-1};\bW;\bW^{-1};\any;\bW^{-1};\bW;\bW^{-1}],\M)^\bW)^{op} , \]
where the objects $x , y \in \M$ determine the double-pointing.  Hence we obtain an equivalence
\[ \word{7}(x,y)^\gpd = (\Funp([\bW^{-1};\bW;\bW^{-1};\any;\bW^{-1};\bW;\bW^{-1}],\M)^\bW)^\gpd \xra{\sim}\hom_{\loc{\M}{\bW}}(x,y) , \]
as desired.
\end{proof}


We now provide a result which was needed in the proof of \cref{connect hom in loc to 7}. 

\begin{lem}\label{htpical 3-arrow calc implies W-op has C3}
If $(\R,\bW) \in \RelCati$ admits a homotopical three-arrow calculus and $\bW \subset \R$ has the two-out-of-three property, then $\bW^{op}$ has property {\Cthree}.
\end{lem}

\begin{proof}
To show that $\bW^{op}$ has property {\Cthree}, we must show that any functor $\pt_\Cati \xra{r^\opobj} \bW^{op}$ (selecting an object $r^\opobj \in \bW^{op}$) has property {\Bthree}, i.e.\! that the induced functor
\[ \bW^{op} \xra{(r^\opobj(\pt_\Cati) \da_3 -)} \Cati \]
has property {\propQ}, i.e.\! that for any map $z^\opobj \xra{\varphi^\opobj} y^\opobj$ in $\bW^{op}$ (opposite to a map $z \xla{\varphi} y$ in $\bW$), the induced map
\[ (r^\opobj(\pt_\Cati) \da_3 z^\opobj) \ra (r^\opobj(\pt_\Cati) \da_3 y^\opobj) \]
is in $\bW_\Thomason \subset \Cati$.  Unwinding the definitions, we can identify this map simply as the functor
\[ \word{3}_{(\bW,\bW)}(r,z) \ra \word{3}_{(\bW,\bW)}(r,y) \]
that postconcatenates a zigzag $r \lwe \bullet \we \bullet \lwe z$ with the map $\varphi$ (considered as a $[\bW^{-1}]$-shaped zigzag) and then composes the last two maps.\footnote{Recall that $\word{z}_3 = (s \ra \bullet \la \bullet \ra t)$ (see \cref{gr:define walking zigzag categories}) while $\word{3} = ( s \lwe \bullet \ra \bullet \lwe t)$, so there are \textit{two} orientation-reversals going on here (counting the passage between $\bW^{op}$ and $\bW$), which cancel each other out.}  Thus, the nerve of the above map in $\Cati$ sits as the upper composite in a commutative square
\[ \begin{tikzcd}[row sep=1.5cm]
\Nervei(\word{3}(r,z)) \arrow{r} \arrow{d}[sloped, anchor=north]{\approx} & \Nervei([\bW^{-1};\any;(\bW^{-1})^{\circ 2}](r,y)) \arrow{r} \arrow{rd} & \Nervei(\word{3}(r,y)) \arrow{d}[sloped, anchor=south]{\approx} \\
\homhamW(r,z) \arrow{rr}{\approx}[swap]{\chi^{\ham(\bW,\bW)}_{r,z,y}(-,\varphi^{-1})} & &\homhamW(r,y)
\end{tikzcd} \]
in $s\S_\KQ$, in which
\begin{itemizesmall}
\item the lower map
\begin{itemizesmall}
\item is the evaluation of the composition map
\[ \homhamW(y,z) \times \homhamW(z,r) \xra{\chi^{\ham(\bW,\bW)}_{z,y,r}} \homhamW(y,r) \]
 in $\ham(\bW,\bW) \in \CatsS$ (recall \cref{hammocks:define space of objects and hom-sspaces}) at the point chosen by the composite
\[ \pt_{s\S} \ra \Nervei([\bW^{-1}](z,y)) \ra \homhamW(z,y) \]
in which the first map is selected by $\varphi$ and the second map is the defining inclusion into the colimit, and
\item lies in $\bW_\KQ \subset s\S$ by \cref{hammocks:hammocks are invt under w.e.},
\end{itemizesmall}
\item the triangle commutes by the definition of the hammock simplicial space as a colimit over $\Z^{op}$ (see \cref{hammocks:define hammocks}),
\item the trapezoid commutes by the definition of composition in the hammock localization (see \cref{hammocks:section hammocks}), and
\item the vertical maps are in $\bW_\KQ$ by the fundamental theorem of homotopical three-arrow calculi (\Cref{hammocks:calculus gives reduction}) since the relative $\infty$-category $(\bW,\bW) \in \RelCati$ admits a homotopical three-arrow calculus by \cref{max relcat on w.e.'s also admits calc}.
\end{itemizesmall}
The upper map is therefore also in $\bW_\KQ$ since $\bW_\KQ \subset s\S$ has the two-out-of-three property, and hence the result follows from \cref{rnerves:groupoid-completion of CSSs}.
\end{proof}

In the proof of \cref{htpical 3-arrow calc implies W-op has C3}, we needed the following stability property of homotopical three-arrow calculi.

\begin{lem}\label{max relcat on w.e.'s also admits calc}
If $(\R,\bW) \in \RelCati$ admits a homotopical three-arrow calculus and $\bW \subset \R$ has the two-out-of-three property, then $(\bW,\bW) \in \RelCati$ also admits a homotopical three-arrow calculus.
\end{lem}

\begin{proof}
This follows directly from \cref{hammocks:define calculus}: if $\bW \subset \R$ has the two-out-of-three property, then the vertical maps in the commutative square
\[ \begin{tikzcd}
 \Funp ( [\bW^{-1} ; \any^{\circ i} ; \any^{\circ j} ; \bW^{-1} ] , \bW)^\bW
\arrow{r} \arrow{d} &
\Funp ( [ \bW^{-1} ; \any^{\circ i} ; \bW^{-1} ; \any^{\circ j} ; \bW^{-1} ] , \bW)^\bW \arrow{d}
\\
\Funp ( [\bW^{-1} ; \any^{\circ i} ; \any^{\circ j} ; \bW^{-1} ] , \R)^\bW
\arrow{r} &
\Funp ( [ \bW^{-1} ; \any^{\circ i} ; \bW^{-1} ; \any^{\circ j} ; \bW^{-1} ] , \R)^\bW
\end{tikzcd} \]
induced by the map $(\bW,\bW) \ra (\R,\bW)$ in $\RelCati$ induce monomorphisms in $\S$ upon groupoid completion.
\end{proof}

\section{Localization of model $\infty$-categories, redux}\label{section model infty-cat gives cSs}

For completeness, we include the following improvement of \cref{rnerve is a SS}, whose proof relies on the fundamental theorem of model $\infty$-categories (\ref{fundamental theorem}).

\begin{thm}\label{rnerve is a CSS}
If $\M$ is a model $\infty$-category with underlying relative $\infty$-category $(\M,\bW)$, then $\NerveRezki(\M,\bW) \in \CSS$, and moreover the morphism $\Nervei(\M) \ra \NerveRezki(\M,\bW)$ in $\CSS$ corresponds to the morphism $\M \ra \loc{\M}{\bW}$ in $\Cati$.
\end{thm}

\begin{proof}
In light of \cref{rnerve is a SS}, it only remains to show that $\NerveRezki(\M,\bW)$ is not just a Segal space, but is in fact complete.  By the calculus theorem (\Cref{hammocks:calculus result}\cref{hammocks:calculus for CSS}), this follows from \cref{model infty-cats are saturated} and the fact that $\bW \subset \M$ satisfies the two-out-of-three property.
\end{proof}

We needed the following result in the proof of \cref{rnerve is a CSS}.

\begin{lem}\label{model infty-cats are saturated}
If $\M$ is a model $\infty$-category, then its underlying relative $\infty$-category $(\M,\bW)$ is saturated.
\end{lem}

\begin{proof}
We would like to show that the localization functor $\M \ra \loc{\M}{\bW}$ creates the subcategory $\bW \subset \M$.  This is equivalent to showing that the functor $\ho(\M) \ra \ho(\loc{\M}{\bW})$ creates the subcategory $\ho(\bW) \subset \ho(\M)$.  For this, we must show that if a map $x \ra y$ in $\ho(\M)$ is taken to an isomorphism in $\ho(\loc{\M}{\bW})$, then it lies in the subcategory $\ho(\bW)$.  By two-out-of-three axiom {\twooutofthreeaxiom}, it suffices to show this in the case that both objects $x,y \in \M^{cf} \subset \M$ are bifibrant.  From here, with \cref{htpy version of fundamental theorem} in hand, the proof runs identically to that of \cite[Theorem 7.8.5]{Hirsch}.
\end{proof}

\bibliographystyle{amsalpha}
\bibliography{fundthm}{}

\providecommand{\bysame}{\leavevmode\hbox to3em{\hrulefill}\thinspace}
\providecommand{\MR}{\relax\ifhmode\unskip\space\fi MR }
\providecommand{\MRhref}[2]{%
  \href{http://www.ams.org/mathscinet-getitem?mr=#1}{#2}
}
\providecommand{\href}[2]{#2}
\begin{thebibliography}{DKS93}

\bibitem[Bara]{BarwickAKT}
C.~Barwick, \emph{{On the algebraic K-theory of higher categories}}, available
  at {\tt arXiv:1204.3607}, v5.

\bibitem[Barb]{BarwickMackeyI}
\bysame, \emph{{Spectral Mackey functors and equivariant algebraic K-theory
  (I)}}, available at {\tt arXiv:1404.0108}, v2.

\bibitem[Bou03]{BousCosimp}
A.~K. Bousfield, \emph{Cosimplicial resolutions and homotopy spectral sequences
  in model categories}, Geom. Topol. \textbf{7} (2003), 1001--1053
  (electronic).

\bibitem[DK80]{DKFunc}
W.~G. Dwyer and D.~M. Kan, \emph{Function complexes in homotopical algebra},
  Topology \textbf{19} (1980), no.~4, 427--440.

\bibitem[DKS93]{DKS-E2}
W.~G. Dwyer, D.~M. Kan, and C.~R. Stover, \emph{An {$E^2$} model category
  structure for pointed simplicial spaces}, J. Pure Appl. Algebra \textbf{90}
  (1993), no.~2, 137--152.

\bibitem[Hir03]{Hirsch}
Philip~S. Hirschhorn, \emph{Model categories and their localizations},
  Mathematical Surveys and Monographs, vol.~99, American Mathematical Society,
  Providence, RI, 2003.

\bibitem[Lur09]{LurieHTT}
Jacob Lurie, \emph{Higher topos theory}, Annals of Mathematics Studies, vol.
  170, Princeton University Press, Princeton, NJ, 2009, also available at {\tt
  http://math.harvard.edu/{$\thicksim$lurie}} and at {\tt arXiv:math/0608040},
  v4.

\bibitem[Lur14]{LurieHA}
\bysame, \emph{Higher algebra}, available at {\tt
  http://www.math.harvard.edu/{$\thicksim$lurie}}, version dated September 14,
  2014.

\bibitem[Man99]{Mandell}
Michael~A. Mandell, \emph{Equivalence of simplicial localizations of closed
  model categories}, J. Pure Appl. Algebra \textbf{142} (1999), no.~2,
  131--152.

\bibitem[May92]{MaySimp}
J.~Peter May, \emph{Simplicial objects in algebraic topology}, Chicago Lectures
  in Mathematics, University of Chicago Press, Chicago, IL, 1992, Reprint of
  the 1967 original.

\bibitem[MGa]{MIC-sspaces}
\emph{{\textup{{A}aron {M}azel-{G}ee,} {\sspacestitle}}}, available at {\tt
  arXiv:1412.8411}, v2.

\bibitem[MGb]{MIC-rnerves}
\emph{{\iftoggle{sspacespaper}{\textup{{A}aron {M}azel-{G}ee,}}{\bysame,}
  {\rnervestitle}}}, available at {\tt arXiv:1510.03150}, v1.

\bibitem[MGc]{MIC-gr}
\emph{{\bysame, {\grtitle}}}, available at {\tt arXiv:1510.03525}, v1.

\bibitem[MGd]{MIC-hammocks}
\emph{{\bysame, {\hammockstitle}}}, available at {\tt arXiv:1510.03961}, v1.

\bibitem[MGe]{MIC-qadjns}
\emph{{\bysame, {\qadjnstitle}}}, available at {\tt arXiv:1510.04392}, v1.

\bibitem[MGq]{adjns}
\emph{{\bysame, {Quillen adjunctions induce adjunctions of quasicategories}}},
  available at {\tt arXiv:1501.03146}, v1.

\end{thebibliography}

\end{document}